\documentclass[11pt,a4paper]{article}

\usepackage[T1]{fontenc}
\usepackage{makeidx}         
\usepackage{graphicx}                       
\usepackage{multicol}       
\usepackage{url}
\usepackage{etoolbox}
\usepackage{caption}
\usepackage{amssymb,amsmath}

\makeatletter
\patchcmd{\maketitle}{\@fnsymbol}{\@alph}{}{} 
\makeatother

\makeatletter
\patchcmd{\@maketitle}{\begin{center}}{\begin{flushleft}}{}{}
\patchcmd{\@maketitle}{\begin{tabular}[t]{c}}{\begin{tabular}[t]{@{}l}}{}{}
\patchcmd{\@maketitle}{\end{center}}{\end{flushleft}}{}{}
\makeatother

\makeatletter
\renewcommand\section{\@startsection{section}{1}{\z@}%
{-3.5ex \@plus -1ex \@minus -.2ex}%
{2.3ex \@plus.2ex}%
{\normalfont\large\bfseries}}
\makeatother

\usepackage[left=3.8cm,right=3.8cm,top=4cm,bottom=4cm]{geometry}

\usepackage{amsthm}
\newtheorem{theorem}{Theorem}
\newtheorem{lemma}{Lemma}
\newtheorem{conjecture}{Conjecture}

\newcommand{\colonequal}{\mathrel{\mathop:}=}
\newcommand{\N}{\mathbb{N}}
\newcommand{\Z}{\mathbb{Z}}

\newcommand{\R}{\mathbb{R}}

\newcommand{\M}{\mathcal{M}}

\usepackage{multicol}
\usepackage{multirow}
\usepackage{tikz}
\usetikzlibrary{math}
\newcommand{\placegrid}{%
\draw[-,thin] (-\d0,0) -- (\s+\d0,0) node[at start, left]{$t$};
\draw[-,thin] (0,-\d0) -- (0,\t+\d0) node[at start, below]{$0$};
\draw[-,thin] (\s,-\d0) -- (\s,\t+\d0) node[at start, below]{$s$};
\draw[-,thin] (-\d0,\t) -- (\s+\d0,\t) node[at start, left]{$1$};
}

\newcommand{\thm}[1]{Theorem~\ref{thm:#1}}

\newcommand{\lem}[1]{Lemma~\ref{lem:#1}}

\newcommand{\fig}[1]{Figure~\ref{fig:#1}}
\renewcommand{\sec}[1]{Section~\ref{sec:#1}} % \sec is predefined as secant

\begin{document}

\font\myfont=cmr12 at 12pt
\title{\begin{flushleft} \textbf{Exact Lower Bounds for Monochromatic Schur Triples and Generalizations} \end{flushleft}}

\author{%
  Christoph Koutschan%
  \thanks{\scriptsize{RICAM, Altenberger Stra\ss e 69, A-4040 Linz, Austria, christoph.koutschan@ricam.oeaw.ac.at}}
  \ and
  Elaine Wong%
  \thanks{\scriptsize{RICAM, Altenberger Stra\ss e 69, A-4040 Linz, Austria, elaine.wong@ricam.oeaw.ac.at
  \rule{40pt}{0pt}
  The final version of this paper appeared in 2020 in the book
  \emph{Algorithmic Combinatorics: Enumerative Combinatorics, Special
    Functions and Computer Algebra}
  of the Springer series \emph{Texts \& Monographs in Symbolic Computation},
  pages 223--248, DOI: 10.1007/978-3-030-44559-1\_13.}}}

% Algorithmic Combinatorics: Enumerative Combinatorics, Special Functions and Computer Algebra, Texts & Monographs in Symbolic Computation, pp. 223–248, 2020. Springer,

\date{}

\maketitle

\vspace{2.2cm}

\begin{flushright}
\begin{minipage}{0.6\textwidth}
  \textit{Dedicated to Peter Paule, our academic father and grandfather.
    Peter, we wish you many more happy, healthy, and productive years.}
\end{minipage}
\end{flushright}
\bigskip

%%%%%%%%%%%%%%%%%%%%%%%%%%%%%%%%%%%%%%%%%%%%%%%%%%%%%%%%%%%%%%%%%%%%%%%%%%%%%%%%%%%%%%%%%%%%%%%
\abstract{ We derive exact and sharp lower bounds for the number of
  monochromatic generalized Schur triples $(x,y,x+ay)$ whose entries are from
  the set $\{1,\dots,n\}$, subject to a coloring with two different
  colors. Previously, only asymptotic formulas for such bounds were known, and
  only for $a\in\N$. Using symbolic computation techniques, these results are
  extended here to arbitrary $a\in\R$. Furthermore, we give exact formulas for
  the minimum number of monochromatic Schur triples for $a=1,2,3,4$, and
  briefly discuss the case $0<a<1$.}
%%%%%%%%%%%%%%%%%%%%%%%%%%%%%%%%%%%%%%%%%%%%%%%%%%%%%%%%%%%%%%%%%%%%%%%%%%%%%%%%%%%%%%%%%%%%%%%

\vspace{0.8cm}

%%%%%%%%%%%%%%%%%%%%%%%%%%%%%%%%%%%%%%%%%%%%%%%%%%%%%%%%%%%%%%%%%%%%%%%%%%%%%%%%%%%%%%%%%%%%%%%
\section{Introduction and historical background}
\label{sec:intro}
%%%%%%%%%%%%%%%%%%%%%%%%%%%%%%%%%%%%%%%%%%%%%%%%%%%%%%%%%%%%%%%%%%%%%%%%%%%%%%%%%%%%%%%%%%%%%%%

\vspace{1mm}\noindent

Let $\N$ denote the set of positive integers. A triple $(x,y,z)\in\N^3$ is
called a Schur triple if its entries satisfy the equation $x+y=z$. The set
$\{1,\dots,n\}$ of all positive integers up to~$n$ will be denoted by $[n]$. A
coloring of $[n]$ is a map $\chi\colon[n]\to C$ for some finite set~$C$ of
colors.  For example, a map $\chi\colon[n]\to\{\mathrm{red},\mathrm{blue}\}$
is a $2$-coloring. We say that a Schur triple is monochromatic (with respect
to a given coloring) if all of its entries have been assigned the same color;
we will abbreviate ``monochromatic Schur triple'' by MST.

With these notations, one can ask questions like: given $n\in\N$ and a coloring~$\chi$
of~$[n]$, how many MSTs are there in $[n]^3$? Let
us denote this number as follows:
\begin{equation}\label{eq:defMnchi}
  \M(n,\chi) \colonequal \bigl|\bigl\{ (x,y,z)\in[n]^3:
  z=x+y \,\land\, \chi(x)=\chi(y)=\chi(z) \bigr\}\bigr|.
\end{equation}

For our purposes, two Schur triples $(x,y,x+y)$ and $(y,x,x+y)$ are considered distinct if
$x\neq y$. We emphasize this convention since sometimes in the literature
these two triples are counted only once, which is equivalent to imposing the
extra condition $x\leq y$. For example, there are exactly four monochromatic Schur
triples on $[6]=\{1,\dots,6\}$ when $2$ and $4$ are colored red and $1,3,5,6$
are colored blue, namely $(1,5,6)$, $(2,2,4)$, $(3,3,6)$, and $(5,1,6)$. We
will use a short-hand notation for $2$-colorings, namely as words on the
alphabet $\{R,B\}$: the $i$-th letter is~$R$ if the integer $i$ is colored red
and~$B$ if it is blue. So the above $2$-coloring would be denoted by
$BRBRBB$. We will also make use of the power notation for words, e.g.,
$R^2B^3=RRBBB$.

The namesake of the triples in this work refers to Issai Schur~\cite{Schur17},
who in 1917 studied a modular version of Fermat's last theorem (first
formulated and proved by Leonard Dickson). In order to
give a simpler proof of the theorem, Schur introduced a \emph{Hilfssatz}
confirming the existence of a least positive integer~$n=n(m)$ such that for any
$m$-coloring of $[n]$ an MST exists (this is nowadays
known as Schur's theorem). In 1927, Van der Waerden~\cite{Vanderwaarden27}
generalized this result to monochromatic arithmetic progressions of any length
$k$. Then in 1928, Ramsey proved his eponymous theorem, showing the existence
of a least positive integer $n$ such that every edge-coloring of a complete
graph on $n$ vertices, with the colors red and blue, admits either a complete
red subgraph or a complete blue subgraph. However, a real increase in the
popularity of these kinds of Ramsey-theoretic problems came with the
rediscovery of Ramsey's theorem in a 1935 paper of Erd\H{o}s and
Szekeres~\cite{ErdosSzekeres35}, which ultimately led to a simpler proof of
Schur's theorem, indicating their close connections. For the curious reader,
this rich history is beautifully depicted in a book by Landman and
Robertson~\cite{LandmanRobertson04}.

We now arrive at a point of more than just questions of existence. In 1959,
Alan Goodman~\cite{Goodman59} studied the \emph{minimum} number of
monochromatic triangles under a 2-edge coloring of a complete graph on $n$
vertices. Then in 1996, Graham, R\"{o}dl, and
Ruci\'{n}ski~\cite{GrahamRodlRucinski96} found it natural to extend the
problem of ``determining the minimum number under any 2-coloring'' to Schur
triples. In fact, Graham offered a prize of 100 USD for an answer to such a
question; it has subsequently been successfully answered many times over, in
an asymptotic sense. In order to give some more context to this problem, we first introduce some additional notation.

We start by wondering about what we can say about the number of MSTs on $[n]$
if we do not prescribe a particular coloring.  It is not difficult to
calculate that there are exactly $\sum_{i=1}^{n-1}i=\frac12n(n-1)=\binom{n}{2}$
Schur triples on~$[n]$. Trivially, this yields an upper bound for the number
of MSTs, which can be achieved by coloring all numbers with the same
color. This is the reason why it is more natural (and more interesting!) to
ask for a lower bound for $\M(n,\chi)$, that is: for given $n\in\N$, what is
the ``best'' lower bound for the number of MSTs regardless of the choice of
coloring? Of course, $0$ is a trivial such lower bound, but we are aiming for something sharp, in the sense that for each~$n$ there exists a coloring for
which this bound is actually attained. Differently stated, we are looking for
the minimal number of monochromatic Schur triples among all possible colorings
of~$[n]$:
\begin{equation}\label{eq:defMn}
  \M(n) \colonequal \min_{\chi\colon[n]\to\{R,B\}} \M(n,\chi).
\end{equation}
For example, for $n=6$, one cannot avoid the occurrence of monochromatic Schur
triples, but there exists a $2$-coloring for which only a single such triple
occurs, namely the triple $(1,1,2)$ for the coloring $RRBBBR$. Therefore, we have
$\M(6)=\M(6,RRBBBR)=1$.

As mentioned before, this problem was only studied from an asymptotic point of
view: Robertson and Zeilberger~\cite{RobertsonZeilberger98} was first to give
the lower bound $\tfrac{1}{22}n^2+O(n)$ as $n\to\infty$ (and consequently won
Graham's cash prize), where it has to be noted that they count only Schur
triples $(x,y,x+y)$ with the condition $x\leq y$ imposed. This lower bound was
independently confirmed by Datskovsky~\cite{Datskovsky03},
Schoen~\cite{Schoen99}, and Thanatipanonda~\cite{Thanatipanonda09}. Schoen
also provided a proof of an ``optimal'' coloring of $[n]$ that would give such
a minimum number, and such a coloring is what we assume later in this
paper. The asymptotic lower bounds for the generalized Schur triples case
$(x,y,x+ay)$ for $a\geq2$ is $\tfrac{1}{2a(a^2+2a+3)}n^2+O(n)$ as $n\to\infty$,
without the requirement of $x\leq y$. This was conjectured by
Thanatipanonda~\cite{Thanatipanonda09} and Butler, Costello, and
Graham~\cite{ButlerCostelloGraham10}, and subsequently proven in 2017 by
Thanatipanonda and Wong~\cite{ThanatipanondaWong17}.

In this paper, we take a slightly different approach by using known computer
algebra techniques and creative simplifications to develop exact formulas for
the minimum number of such triples (in both the Schur triples case and the generalized
Schur triples case) and give an analysis of the transitional behavior between
the cases. Thus, in order to keep some consistency for comparison, we will
remove the assumption of $x\leq y$ when counting MSTs. In this way, we can
explain why the behavior of the minimum number of triples jumps when moving
from the case $a=1$ to the case $a\geq2$ (note that the above asymptotic
formula does not specialize to the expected prefactor~$\frac{1}{11}$ when
$a=1$ is substituted).

The overall plan is to systematically exploit the full force of symbolic
computation and perform a complete analysis of determining the minimum number
of monochromatic triples $(x,y,x+ay)$ in both the discrete context $(a\in\N)$
and the continuous context $(a\in\mathbb{R}^+)$. This requires three courses
of a mathematical meal. We serve an appetizer in \sec{schur}, showing how to
derive an exact formula for the minimum in the classic Schur triple case
(corresponding to $a=1$ in the general equation). This sets us up for the main
course in \sec{real}, where we perform a full analysis for $a>0$, illustrating
that a global minimum can always be found. Interesting transitional
behaviors occur at many locations for $a\in(0,1)$ and one key transition
occurs at $a\approx 1.17$.  Admittedly, this course may be a bit difficult to
swallow, and we hope that the reader will not suffer from indigestion.  For
dessert, we follow the procedure described in \sec{schur}, and illustrate how
it can systematically produce (ostensibly, an infinite number) of exact
formulas for the minimum number of generalized Schur triples. Accordingly, in
\sec{discrete}, we leave the reader with exact formulas for the minimum number
of generalized Schur triples for $a=2,3,4$, and $a=\frac12$, with the hope that
s/he will leave satisfied.

\pagebreak

For the reader's convenience, all computations and diagrams are in the
Mathematica notebook \cite{KoutschanWong19} that accompanies this paper,
freely available at the first author's website. This material may also be of
independent interest, since we believe that also other problems can be
attacked in a similar fashion, see for example the recent study on the
peacable queens problem~\cite{YaoZeilberger19}.

\vspace{0.8cm}

%%%%%%%%%%%%%%%%%%%%%%%%%%%%%%%%%%%%%%%%%%%%%%%%%%%%%%%%%%%%%%%%%%%%%%%%%%%%%%%%%%%%%%%%%%%%%%%
\section{Exact lower bound for monochromatic Schur triples}
\label{sec:schur}
%%%%%%%%%%%%%%%%%%%%%%%%%%%%%%%%%%%%%%%%%%%%%%%%%%%%%%%%%%%%%%%%%%%%%%%%%%%%%%%%%%%%%%%%%%%%%%%

\vspace{1mm}\noindent

\begin{figure}[t]
  \begin{center}
    \includegraphics[width=0.5\textwidth]{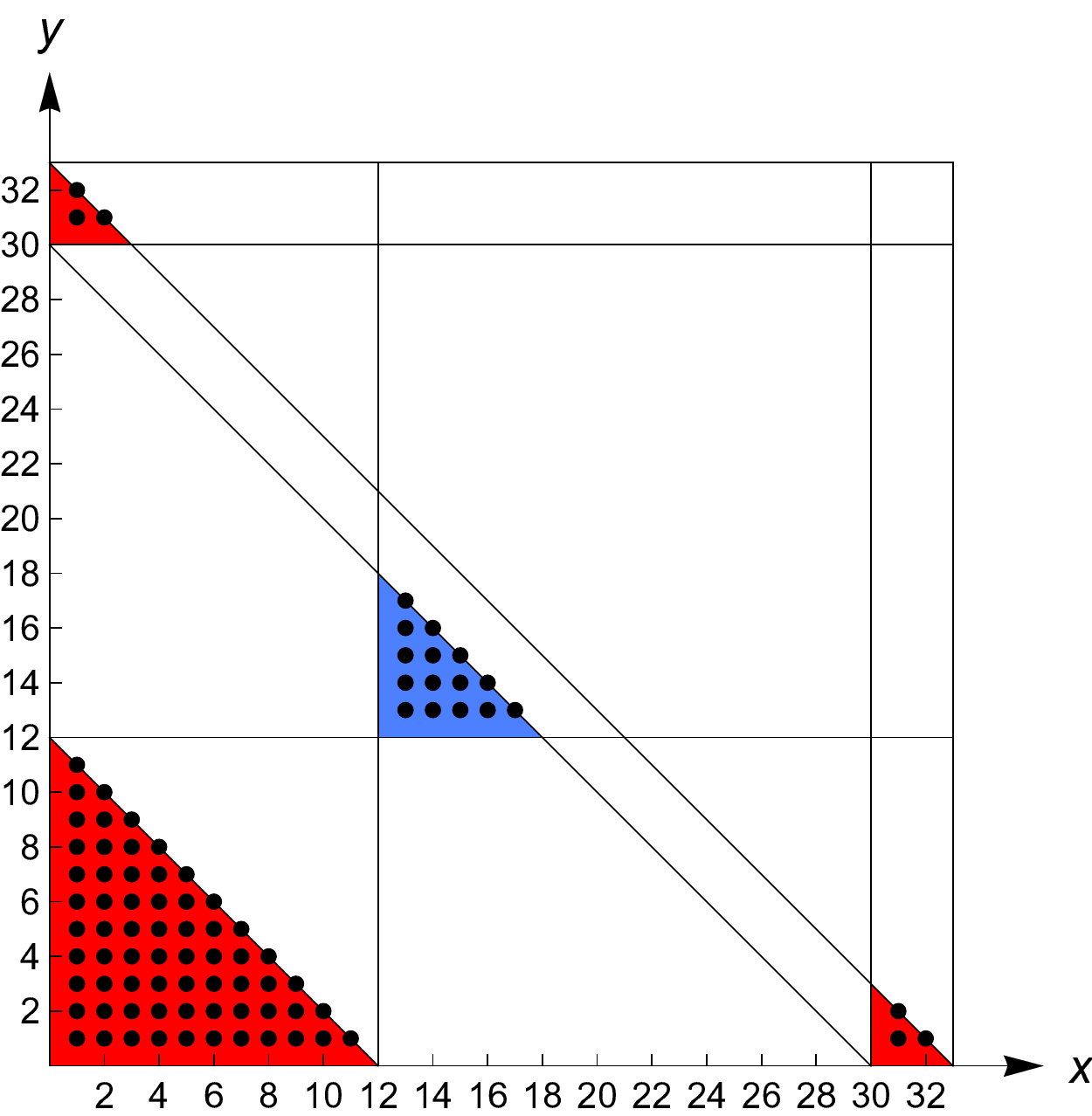}
  \end{center}
  \caption{All $\M(33)=87$ monochromatic Schur triples for $s=12$ and $t=30$
    with corresponding coloring
    $R^{12}B^{18}R^3$; each triple $(x,y,x+y)$ is represented by a dot at
    position $(x,y)$. The vertical lines are given by $x=s$, $x=t$, and $x=n$,
    the horizonal ones by $y=s$, $y=t$, and $y=n$. The three diagonal lines
    visualize the equations $x+y=s$, $x+y=t$, and $x+y=n$.}
  \label{fig:MMST33}
\end{figure}

It has been shown previously~\cite{RobertsonZeilberger98,Schoen99} that for
fixed~$n$ the number $\M(n,\chi)$ is minimized when $\chi$ consists of three
blocks of numbers with the same color (``runs''), i.e., when $\chi$ is of the form
$R^sB^{t-s}R^{n-t}$, where $s$ and $t$ are approximately $\frac{4}{11}n$ and
$\frac{10}{11}n$, respectively. In this section, we derive exact expressions
for the optimal choice of $s$ and~$t$, as well as for the corresponding
minimum~$\M(n)$.

\begin{lemma}\label{lem:NoMST}
  Let $n,s,t\in\N$ be such that $1\leq s\leq t\leq n$. Moreover, assume that the
  inequalities $t\geq 2s$ and $s\geq n-t$ hold. Then the number of monochromatic
  Schur triples on $[n]$ under the coloring $R^sB^{t-s}R^{n-t}$, denoted by
  $\M(n,s,t)$, is exactly
  \begin{equation}\label{eq:NoMST}
    \M(n,s,t) = \frac{s(s-1)}{2} + \frac{(t-2s)(t-2s-1)}{2} + (n-t)(n-t-1).
  \end{equation}
\end{lemma}
\begin{proof}
  In \fig{MMST33} the situation is depicted for $n=33$, $s=12$, and $t=30$.
  One sees that the dots representing the MSTs are arranged in four regions of
  right triangular shape.  The triangles arise as follows:
  \begin{enumerate}
  \item The dots in the lower left corner correspond to red MSTs,
    whose components are taken from the first block of red numbers; hence there
    are $s-1$ dots in the first row of this triangle.
  \item The central triangle contains all blue MSTs, whose first two
    components $(x,y)$ satisfy the inequalities $x>s$, $y>s$, and $x+y\leq t$.
    Note that such MSTs only exist if $t\geq 2s+2$ (for $t=2s+1$ and $t=2s$ the
    second term in~\eqref{eq:NoMST} vanishes and the formula is still correct).
    The number of dots on each side is therefore $t-2s-1$.
  \item The two triangles in the upper left and lower right corners correspond
    to red MSTs, whose first two entries belong to different blocks of red
    numbers.  By symmetry they have the same shape and they have $n-t-1$ dots
    on their sides. Here we use the condition $s\geq n-t$, because otherwise
    these two regions would no longer be triangles and we would be counting
    different things beyond the scope of our assumptions.
  \end{enumerate}
  Adding up the contributions from these three cases, one obtains the claimed
  formula.
\end{proof}

The optimal values for $s$ and $t$ are easily derived using the techniques of multivariable
calculus, once the form $R^sB^{t-s}R^{n-t}$ is assumed: by letting $n$ go to
infinity and by scaling the square $[0,n]^2\subset\R^2$ to the unit
square~$[0,1]^2$, we see that the portion of pairs $(x,y)\in[n]^2$ for which
$(x,y,x+y)$ is an MST among all pairs in $[n]^2$ equals
the area of a certain region in the unit square; for example, see the shaded regions in \fig{MMST33}. In this
limit process, the integers $s$ and $t$ turn into real numbers satisfying
$0\leq s\leq t\leq 1$. According to~\eqref{eq:NoMST} the area of the shaded
region in \fig{MMST33} is given by the formula
\[
  A(s,t) = \frac{s^2}{2} + \frac{(t - 2 s)^2}{2} + 2\cdot\frac{(1 - t)^2}{2} =
  \frac{5s^2}{2} + \frac{3t^2}{2} - 2st - 2t + 1.
\]
Equating the gradient
\[
  \left(\frac{\partial A}{\partial s}, \frac{\partial A}{\partial t}\right) =
  (5s-2t, 3t-2s-2)
\]
to zero, one immediately gets the location of the minimum
$(s,t)=\bigl(\frac{4}{11},\frac{10}{11}\bigr)$.

\begin{lemma}\label{lem:minMnst}
  For fixed $n\in\N$, the integers $s_0$ and $t_0$ that minimize the function
  $\M(n,s,t)$ are given by
  \[
    s_0=\Bigl\lfloor\frac{4n+2}{11}\Bigr\rfloor
    \qquad\text{and}\qquad
    t_0=\Bigl\lfloor\frac{10n}{11}\Bigr\rfloor.
  \]
\end{lemma}
\begin{proof}
  Strictly speaking, we prove the minimality of the function $\M(n,s,t)$
  under the additional assumption $t\geq 2s\land s\geq n-t$ from \lem{NoMST}.
  The fact that this is also the global minimum for all $1\leq s\leq t\leq n$
  follows as a special case from the more general discussion as described in
  the proof of \lem{min}.
  
  The statement is proven by case distinction into $11$ cases, according to
  the remainder $n$ modulo~$11$. Here we show details for the case
  $n=11k+5$, and the remaining cases can be similarly verified 
  with a computer; for these cases we refer the reader to the
  accompanying electronic material~\cite{KoutschanWong19}.

  By setting $n=11k+5$ we can eliminate the floors from the definitions of $s_0$
  and~$t_0$; we obtain $s_0=\lfloor \frac{1}{11}(4n+2)\rfloor=4k+2$ and
  $t_0=\lfloor\frac{10}{11}n\rfloor=10k+4$. Our goal is to show that among all
  integers $i,j\in\Z$ the expression $\M(n,s_0+i,t_0+j)$ is minimal for $i=j=0$.
  Using~\eqref{eq:NoMST} one gets
  \[
    \M(11k+5, 4k+2+i, 10k+4+j) =
    \frac12 \bigl(2 + 5 i + 5 i^2 - 3 j - 4 i j + 3 j^2 + 12 k + 22 k^2\bigr).
  \]
  The stated goal is equivalent to showing that the polynomial
  \[
    p(i,j) = 5 i + 5 i^2 - 3 j - 4 i j + 3 j^2
  \]
  is nonnegative for all $(i,j)\in\Z^2$. Such a task can, in principle, be
  routinely executed by cylindrical algebraic decomposition
  (CAD)~\cite{Collins75}. In this method, the variables $i$ and $j$ are
  treated as real variables, which causes some problems in the present
  application. The reason is that $p(i,j)\geq0$ does not hold for all
  $i,j\in\R$.  The situation is depicted in \fig{MinMij}, where the ellipse
  represents the zero set of $p(i,j)$ and its inside consists of values $(i,j)$ for
  which the polynomial~$p(i,j)$ is negative. To our relief, we see that no
  integer lattice points lie inside the ellipse, since such points would be
  counterexamples to our claim.

  Our strategy now is the following: we prove that $p(i,j)\geq0$ for all
  integer points that are close to $(0,0)$, e.g., for all $(i,j)$ with $-2\leq
  i\leq2$ and $-2\leq j\leq2$. These points are shown in \fig{MinMij}, with
  the respective value of $p(i,j)$ attached to them. In particular, we see
  that the minimum $p(i,j)=0$ is attained several times, namely on the three
  points that lie exactly on the boundary of the ellipse.

  Then we invoke cylindrical algebraic decomposition on the formula
  \begin{equation}\label{eq:cadinput}
    \forall i,j\in\R\colon
    (-2 \leq i \leq 2 \land -2 \leq j \leq 2) \lor p(i,j) \geq 0,
  \end{equation}
  which states that if the point $(i,j)$ lies outside the square that we have
  already considered, then $p(i,j)\geq0$ holds. Calling the Mathematica
  command \texttt{CylindricalDecomposition} with input~\eqref{eq:cadinput}, we
  immediately get \texttt{True}.
\end{proof}

\begin{figure}[t]
  \begin{center}
    \includegraphics[width=0.4\textwidth]{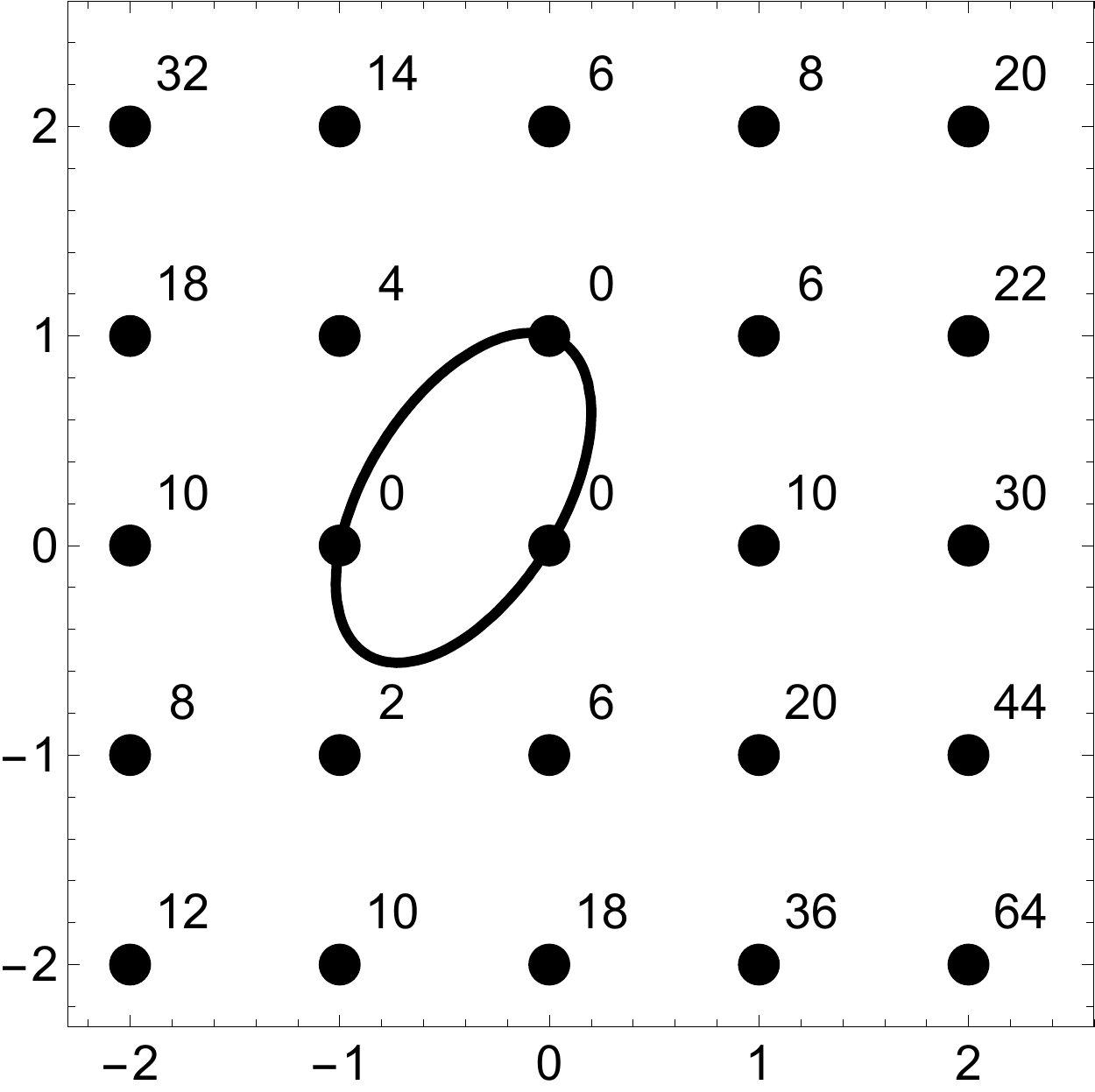}
  \end{center}
  \caption{Zero set of the polynomial $p(i,j)$ from \lem{minMnst} and its
    values at integer lattice points $(i,j)\in\Z^2$.}
  \label{fig:MinMij}
\end{figure}

We are ready to state the main theorem of this section, which is an exact
formula for the minimal number of MSTs for any $2$-coloring of~$[n]$. Apart
from the asymptotic results mentioned in \sec{intro}, there is only one
paper~\cite{Schoen99} where a similar result is stated, but only for the case $n=22k$ and for Schur triples $(x,y,x+y)$ with $x\leq y$.
In contrast, we consider all $x,y\in[n]$ and our formula holds for all
$n\in\N$.

\begin{theorem}\label{thm:minMn}
  The minimal number of monochromatic Schur triples that can be attained
  under any $2$-coloring of $[n]$ is
  \[
    \M(n) = \Bigl\lfloor\frac{n^2 - 4 n + 6}{11}\Bigr\rfloor.
  \]
\end{theorem}
\begin{proof}
  As in \lem{minMnst}, we argue by case distinction $n=11k+\ell$,
  $0\leq\ell\leq10$. Using $s_0=\lfloor \frac{1}{11}(4n+2)\rfloor$ and
  $t_0=\lfloor\frac{10}{11}n\rfloor$ from the lemma, we obtain the following
  values for $\M(n,s_0,t_0)$:
  \begin{alignat*}{3}
    \ell=0\colon & \M(11 k,4 k,10 k) &&= 11 k^2-4 k &&=
    \tfrac{1}{11}(n^2-4 n) \\
    \ell=1\colon & \M(11 k+1,4 k,10 k) &&= 11 k^2-2 k &&=
    \tfrac{1}{11}(n^2-4 n+3) \\
    \ell=2\colon & \M(11 k+2,4 k,10 k+1) &&= 11 k^2 &&=
    \tfrac{1}{11}(n^2-4 n+4) \\
    \ell=3\colon & \M(11 k+3,4 k+1,10 k+2) &&= 11 k^2+2 k &&=
    \tfrac{1}{11}(n^2-4 n+3) \\
    \ell=4\colon & \M(11 k+4,4 k+1,10 k+3) &&= 11 k^2+4 k &&=
    \tfrac{1}{11}(n^2-4 n) \\
    \ell=5\colon & \M(11 k+5,4 k+2,10 k+4) &&= 11 k^2+6 k+1 &&=
    \tfrac{1}{11}(n^2-4 n+6) \\
    \ell=6\colon & \M(11 k+6,4 k+2,10 k+5) &&= 11 k^2+8 k+1 &&=
    \tfrac{1}{11}(n^2-4 n-1) \\
    \ell=7\colon & \M(11 k+7,4 k+2,10 k+6) &&= 11 k^2+10 k+2 &&=
    \tfrac{1}{11}(n^2-4 n+1) \\
    \ell=8\colon & \M(11 k+8,4 k+3,10 k+7) &&= 11 k^2+12 k+3 &&=
    \tfrac{1}{11}(n^2-4 n+1) \\
    \ell=9\colon & \M(11 k+9,4 k+3,10 k+8) &&= 11 k^2+14 k+4 &&=
    \tfrac{1}{11}(n^2-4 n-1) \\
    \ell=10\colon & \M(11 k+10,4 k+3,10 k+9) &&= 11 k^2+16 k+6 &&=
    \tfrac{1}{11}(n^2-4 n+6)
  \end{alignat*}
  One easily observes that in each case, the result is of the form
  $\tfrac{1}{11}(n^2-4 n)+\delta_{\ell}$, where $-\tfrac{1}{11}\leq\delta_{\ell}\leq\tfrac{6}{11}$
  holds for all~$\ell$. Hence the claimed formula follows.
\end{proof}
The first $25$ terms of the sequence $\bigl(\M(n)\bigr)_{n\geq1}$ are
\[
  0, 0, 0, 0, 1, 1, 2, 3, 4, 6, 7, 9, 11, 13, 15, 18, 20, 23, 26, 29,
  33, 36, 40, 44, 48, \dots
\]
We have added this sequence to the Online Encyclopedia of Integer
Sequences~\cite{Sloane} under the number A321195.

\vspace{0.8cm}

%%%%%%%%%%%%%%%%%%%%%%%%%%%%%%%%%%%%%%%%%%%%%%%%%%%%%%%%%%%%%%%%%%%%%%%%%%%%%%%%%%%%%%%%%%%%%%%
\section{\hspace{-4pt} Asymptotic lower bound for generalized Schur triples}
\label{sec:real}
%%%%%%%%%%%%%%%%%%%%%%%%%%%%%%%%%%%%%%%%%%%%%%%%%%%%%%%%%%%%%%%%%%%%%%%%%%%%%%%%%%%%%%%%%%%%%%%

\vspace{1mm}\noindent

We now turn to generalized Schur triples, i.e., triples $(x,y,z)$ subject to
$z=x+ay$ for some parameter $a\in\N$, as studied by Thanatipanonda and
Wong~\cite{ThanatipanondaWong17}. Here, we allow $a$ to be even more general,
i.e., $a\in\R^+$. Consequently, we have to adapt the definition of generalized
Schur triples: we use the condition $z=x+\lfloor ay\rfloor$. The case $a<0$
does not add new aspects to the analysis, as it can be transformed to the $a>0$
case by exchanging the roles of $x$ and $z$ and by changing the floor function
to a ceiling.

Again, we choose to use the assumption that the minimal number of
monochromatic generalized Schur triples (MGSTs) occurs at a coloring in the
form of three blocks $R^sB^{t-s}R^{n-t}$. We justify using this assumption
with the experimental evidence of Butler, Costello, and
Graham~\cite{ButlerCostelloGraham10} (who argued for the generalized Schur
triple case $a>1$) and adapting the intuition in the argument of
Schoen~\cite{Schoen99} (who only argued for the Schur triple case $a=1$).

We would like to know for which choice of $s$ and $t$ (depending on~$n$
and~$a$) the minimum occurs. Similar to the previous section, we let $n$ go to
infinity and correlate the number of MGSTs with the area of polygonal regions
in the unit square. We then define a function $A(s,t,a)$ that determines this
area, and minimize it. Hence, throughout this section, $s$ and $t$ are real
numbers with $0\leq s\leq t\leq1$.

\fig{MGSTarea} shows two situations for different choices of $a,s,t$.  In
contrast to the previous section, we do a very careful case analysis and do
not impose extra conditions on $s$ and~$t$ as in \lem{NoMST}, at the cost of
introducing a ``few'' more case distinctions. The full case analysis for normal
Schur triples then follows by specializing to $a=1$ in the resulting formulas.

\begin{figure}[t]
  \begin{center}
    \includegraphics[width=0.4\textwidth]{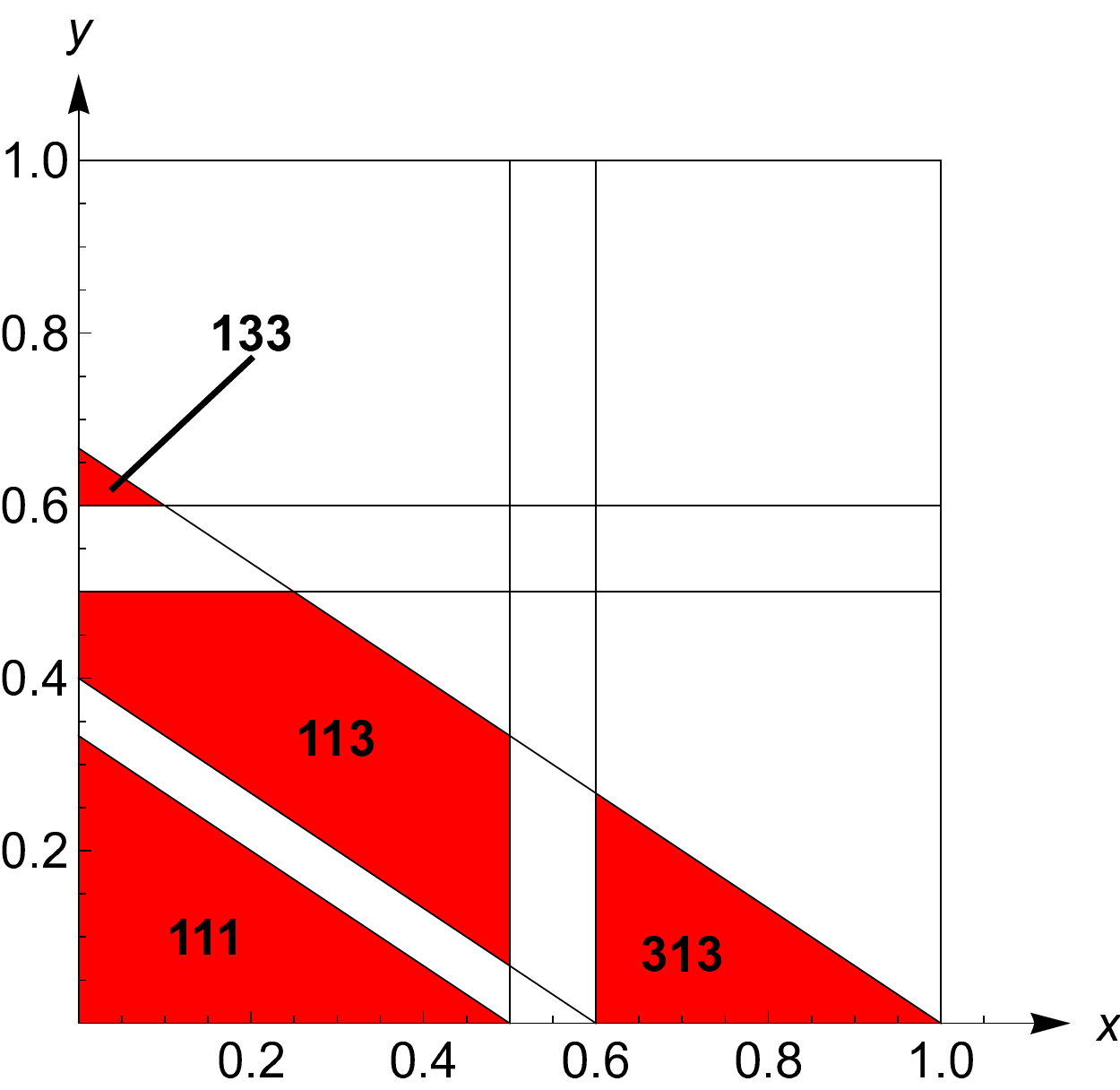}
    \qquad\qquad
    \includegraphics[width=0.4\textwidth]{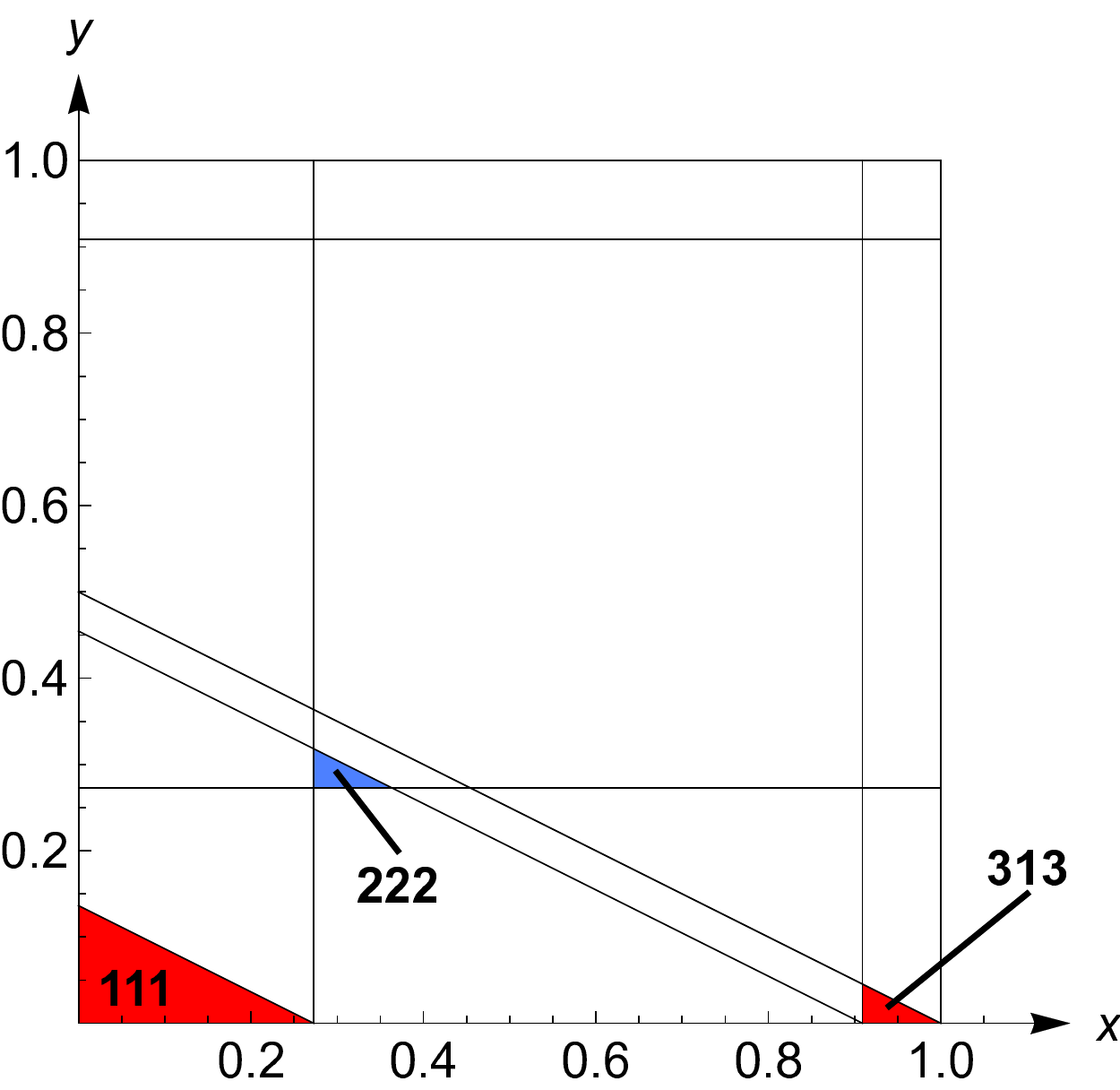}  
  \end{center}
  \caption{Regions (in red and blue) corresponding to monochromatic generalized Schur
    triples for $a=\frac32$, $s=\frac12$, $t=\frac35$ (left) and
    $a=2$, $s=\frac{3}{11}$, $t=\frac{10}{11}$ (right); their area being measured
    by $A(s,t,a)$ from \lem{MGSTarea}.}
  \label{fig:MGSTarea}
\end{figure}

In the process of analyzing the different cases, we encounter several
conditions on $a,s,t$. For our referencing convenience, we
distinguish these conditions here using the following abbreviations:
\begin{equation}\label{eq:pwconds}
  \begin{aligned}
    C_1 &\;\equiv\; 1-as\geq 0, & \qquad\qquad
    C_2 &\;\equiv\; 1-as-s\geq 0, \\
    C_3 &\;\equiv\; 1-as-t\geq 0, &
    C_4 &\;\equiv\; t-as\geq 0, \\
    C_5 &\;\equiv\; t-as-s\geq 0, &
    C_6 &\;\equiv\; 1-at\geq 0, \\
    C_7 &\;\equiv\; 1-at-s\geq 0, &
    C_8 &\;\equiv\; 1-at-t\geq 0, \\
    C_9 &\;\equiv\; 1-a\geq 0, &
    C_{10} &\;\equiv\; 1-a-s\geq 0, \\
    C_{11} &\;\equiv\; s-a\geq 0, &
    C_{12} &\;\equiv\; 1-a-t\geq 0, \\
    C_{13} &\;\equiv\; t-a\geq 0, &
    C_{14} &\;\equiv\; t-a-s\geq 0, \\
    C_{15} &\;\equiv\; s-a t\geq 0, &
    C_{16} &\;\equiv\; t-a t-s\geq 0.
  \end{aligned}
\end{equation}

In Figures~\ref{fig:pwregions1} and~\ref{fig:pwregions2}, the lines that
represent some of these conditions are depicted. They split the triangle
$0\leq s\leq t\leq1$ into several regions, depending on the value of~$a$.

\begin{lemma}\label{lem:MGSTarea}
  Let $a,s,t\in\R$ with $a>0$ and $0\leq s\leq t\leq1$. Then the area
  $A(s,t,a)$ of the region
  \[
    \bigl\{(x,y)\in\R^2 :
    (x,y,x+ay)\in\bigl([0,s]\cup(t,1]\bigr){}^3 \lor (x,y,x+ay)\in(s,t]^3 \bigr\}
  \]
  is given by a piecewise defined function, where $70$
  case distinctions have to be made.  For the sake of brevity, only the first $17$ cases are
  listed below, since they will be the most important ones in the subsequent
  analysis; in fact they are sufficient to describe $A(s,t,a)$ for
  $a\geq1$. We label the region corresponding to the $i$-th case as $(R_i)$.
  They are expressed in terms of the conditions \eqref{eq:pwconds}
  (where overlines denote negations):
  \[
  \begin{array}{@{}l@{\hspace{0.2cm}}l@{\hspace{0.6cm}}l@{}}
    & \text{conditions on } a,s,t & A(s,t,a) \\ \hline \rule{0pt}{16pt}
 (R_1) &
   \overline{C_1} &
   \frac{s^2-2 t s+2 s+t^2-2 t+1}{2 a} \\[1.5ex]
 (R_2) &
   C_3\land C_4\land \overline{C_6} &
   \frac{2 a s^2+2 s^2+2 a s-4 a t s-2 t s+t^2}{2 a} \\[1.5ex]
 (R_3) &
   C_3\land \overline{C_4}\land \overline{C_6} &
   \frac{-a^2 s^2+2 a s^2+2 s^2+2 a s-2 a t s-2 t s}{2 a} \\[1.5ex]
 (R_4) &
   \overline{C_2}\land C_4\land \overline{C_6} &
   \frac{s^2+2 a s-2 a t s-2 t s+2 s+2 t^2-2 t}{2 a} \\[1.5ex]
 (R_5) &
   \overline{C_2}\land \overline{C_4}\land C_6 &
   \frac{-a^2 s^2+s^2+2 a s-2 t s+2 s+a^2 t^2+t^2-2 a t-2 t+1}{2 a} \\[1.5ex]
 (R_6) &
   C_1\land \overline{C_2}\land \overline{C_4}\land \overline{C_6} &
   \frac{-a^2 s^2+s^2+2 a s-2 t s+2 s+t^2-2 t}{2 a} \\[1.5ex]
 (R_7) &
   C_2\land \overline{C_3}\land C_4\land \overline{C_6} &
   \frac{a^2 s^2+2 a s^2+2 s^2-2 a t s-2 t s+2 t^2-2 t+1}{2 a} \\[1.5ex]
 (R_8) &
   C_2\land \overline{C_3}\land \overline{C_4}\land C_6 &
   \frac{2 a s^2+2 s^2-2 t s+a^2 t^2+t^2-2 a t-2 t+2}{2 a} \\[1.5ex]
 (R_9) &
   C_2\land \overline{C_3}\land \overline{C_4}\land \overline{C_6} &
   \frac{2 a s^2+2 s^2-2 t s+t^2-2 t+1}{2 a} \\[1.5ex]
 (R_{10}) &
   C_3\land C_4\land C_6\land \overline{C_7} &
   \frac{2 a s^2+2 s^2+2 a s-4 a t s-2 t s+a^2 t^2+t^2-2 a t+1}{2 a} \\[1.5ex]
 (R_{11}) &
   C_3\land \overline{C_4}\land C_6\land \overline{C_7} &
   \frac{-a^2 s^2+2 a s^2+2 s^2+2 a s-2 a t s-2 t s+a^2 t^2-2 a t+1}{2 a} \\[1.5ex]
 (R_{12}) &
   \overline{C_4}\land C_8 &
   \frac{(1+2a-a^2) s^2+2 s (1-2 a t+a-t)+(a t+t-1)^2}{2 a} \\[1.5ex]
   %\frac{-a^2 s^2+2 a s^2+s^2+2 a s-4 a t s-2 t s+2 s+a^2 t^2+2 a t^2+t^2-2 a t-2 t+1}{2 a} \\[1.5ex]
 (R_{13}) &
   \overline{C_4}\land C_7\land \overline{C_8} &
   \frac{-a^2 s^2+2 a s^2+s^2+2 a s-4 a t s-2 t s+2 s}{2 a} \\[1.5ex]
 (R_{14}) &
   C_4\land C_8\land \overline{C_9} &
   \frac{(a^2+2a+2) t^2-2 t (3 a s+a+s+1)+(s+1) (2 a s+s+1)}{2 a} \\[1.5ex]
   %\frac{2 a s^2+s^2+2 a s-6 a t s-2 t s+2 s+a^2 t^2+2 a t^2+2 t^2-2 a t-2 t+1}{2 a} \\[1.5ex]
 (R_{15}) &
   C_4\land C_7\land \overline{C_8}\land \overline{C_9} &
   \frac{2 a s^2+s^2+2 a s-6 a t s-2 t s+2 s+t^2}{2 a} \\[1.5ex]
 (R_{16}) &
   \overline{C_2}\land C_4\land C_6\land \overline{C_9} &
   \frac{s^2+2 a s-2 a t s-2 t s+2 s+a^2 t^2+2 t^2-2 a t-2 t+1}{2 a} \\[1.5ex]
 (R_{17}) &
   C_2\land \overline{C_3}\land C_4\land C_6\land \overline{C_9} &
   \frac{a^2 s^2+2 a s^2+2 s^2-2 a t s-2 t s+a^2 t^2+2 t^2-2 a t-2 t+2}{2 a}
  \end{array}
  \]
\end{lemma}
\begin{proof}
  As can be seen in \fig{MGSTarea}, the region whose area we would like to
  determine is the union of several polygons. Let $I_1=[0,s]$, $I_2=(s,t]$, and
  $I_3=(t,1]$ denote the intervals that correspond to the different blocks of
  the coloring ($I_1$ and $I_3$ being red and $I_2$ being blue).  Then $x,y\in
  I_1 \land x+ay\in I_3$ is allowed while $x,y\in I_1 \land x+ay\in I_2$ is
  not. From this point on, we will refer to the case $(x,y,x+ay)\in I_i\times I_j\times I_k$ by
  $ijk$. It is easy to see that we have to consider only seven cases: $111$,
  $222$, $113$, $131$, $133$, $313$, $333$.  The cases $311$ and $331$ are
  clearly impossible since $x\geq t$ contradicts $x+ay\leq s$. All other
  combinations of $1,2,3$ violate the monochromatic coloring condition.

  In both parts of \fig{MGSTarea}, case $111$ corresponds to the triangle that
  touches the origin.  The coordinates of its other two vertices are $(s,0)$
  and $(0,\frac{s}{a})$, hence its area is $\frac12\cdot
  s\cdot\frac{s}{a}$. However, this is valid only for $a\geq1$. If $a<1$, then
  the point $(0,\frac{s}{a})$ is above the line $y=s$ and so the top of the
  triangle is cut off. As a result, one obtains a quadrilateral with vertices
  $(0,0)$, $(s,0)$, $(s-as,s)$, $(0,s)$, whose area is given by $\frac12\cdot
  s\cdot(2s-as)$.

  The case $222$ is similar, with the difference being that the corresponding
  polygon disappears if $\frac{t-s}{a}<s$; in the right part of
  \fig{MGSTarea} the polygon $222$ is present while in the left part it is
  not. The polygons $313$, $333$, and $131$ are characterized by comparably
  simple case distinctions, while $133$ and $113$ require a much more involved
  analysis. In \fig{cases133}, we present such an analysis for $133$, and
  refer to the accompanying electronic material~\cite{KoutschanWong19} for $113$.

  What we have achieved so far is a representation of $A(s,t,a)$ as a sum of
  seven piecewise functions. However, what is required is a representation of
  $A(s,t,a)$ as a single piecewise function, since that will be needed for
  determining the location of the minimum.

  The conditions that are used to characterize the different pieces in
  \fig{cases133} (and in the remaining cases that have not
  been discussed explicitly), are listed in~\eqref{eq:pwconds}. In order to combine the
  seven piecewise functions, we need a common refinement of the regions on
  which they are defined. We start with the finest possible refinement, which
  is obtained by considering all $2^{16}=65536$ logical combinations
  of $C_i$ and $\overline{C_i}$ for $1\leq i\leq16$.
  Using Mathematica's simplification procedures, we
  remove those cases that contain contradictory combinations of conditions,
  such as $C_1\land\overline{C_2}$ for example. After this purging, we are
  left with a subdivision of the set
  \begin{equation}\label{eq:prism}
    \{(s,t,a) : 0\leq s\leq t\leq 1 \land a\geq0\} \subset \R^3,
  \end{equation}
  which is an infinite triangular prism, into $114$ polyhedral
  regions. Finally, we merge regions on which $A(s,t,a)$ is defined by the
  same expression into a single region, yielding a representation of
  $A(s,t,a)$ as a piecewise function defined by $70$ different expressions.
  Each of them is of the form $\frac{1}{a}p(s,t,a)$ where $p$ is a polynomial
  in $s,t,a$ of degree at most~$2$ in each of the variables. For more details,
  and to see the definition of $A(s,t,a)$ in its full glory,
  see the accompanying electronic material~\cite{KoutschanWong19}.
\end{proof}

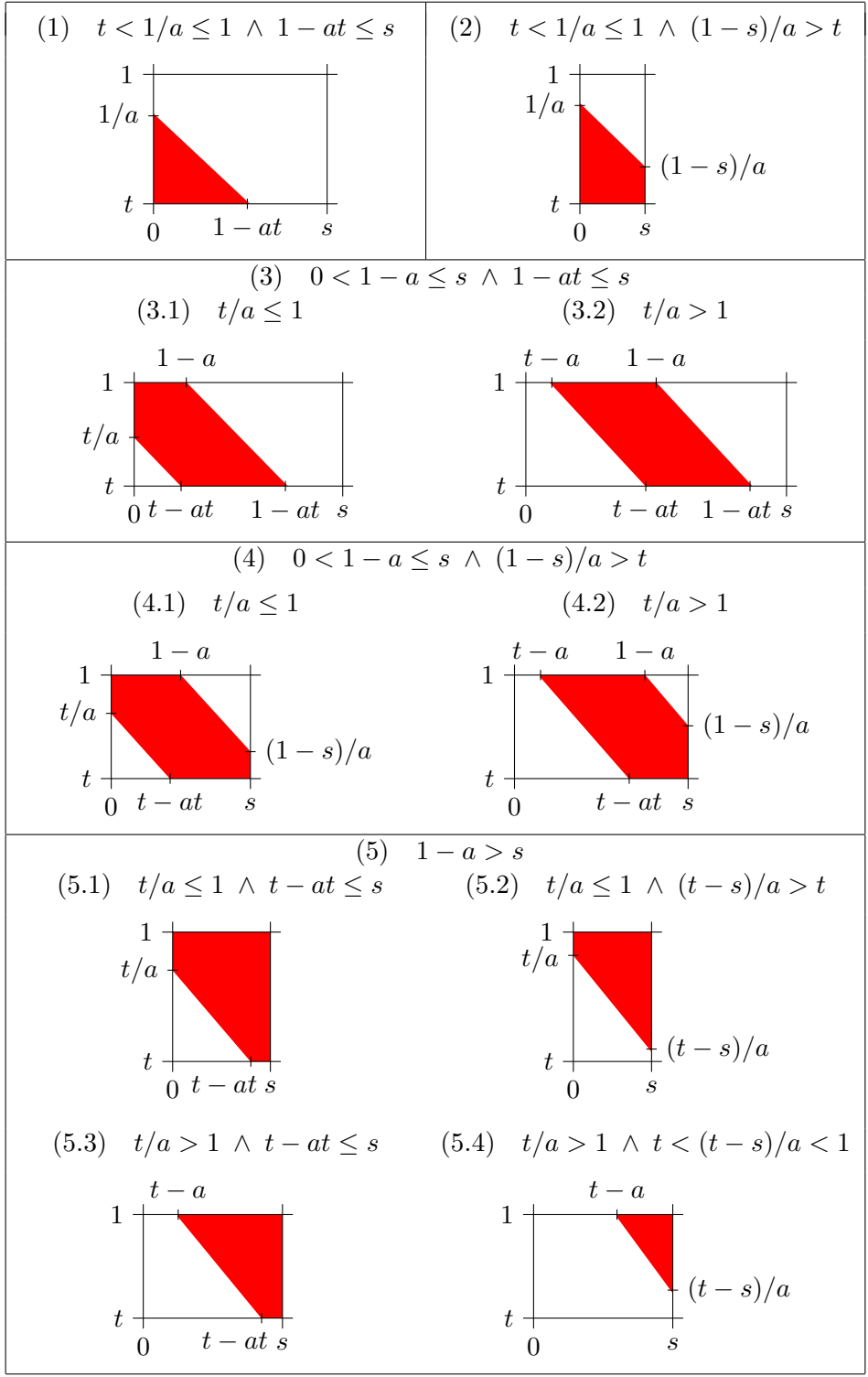
\begin{figure}
\centering

\begin{tabular}{|cc|}

\hline

\multicolumn{1}{|c|}{\rule{160pt}{0pt}} & \rule{160pt}{0pt} \\[-2ex]

\multicolumn{1}{|c|}{
% (1) condition
% $(1) \enspace a\geq 1 \;\land\; 1-at \leq s \;\land\; t < 1/a$ } &
$(1) \quad  t < 1/a \leq 1 \;\land\; 1-at \leq s$ } &
% (2) condition
%$(2) \enspace a\geq 1 \;\land\; t < (1-s)/a \;\land\; t < 1/a$
$(2) \quad t < 1/a \leq 1 \;\land\; (1-s)/a > t$
\\[1ex]
\multicolumn{1}{|c|}{
% (1) graph
\begin{tikzpicture}[scale=0.75]
\tikzmath{\s=3.33; \t=2.5; \d0=0.2; \d1=0.02; \d2=0.1; \p1=1.7; \p2=1.8;}
\coordinate (A) at (\d1,\d1);
\coordinate (B) at (\p2,\d1);
\coordinate (C) at (\d1,\p1);
\draw[-,red,thick,fill] (A)--(B)--(C)--(A);
\placegrid
\draw[-,thin] (\p2,-\d2) -- (\p2,\d2) node[at start,below]{$1-at$};
\draw[-,thin] (-\d2,\p1) -- (\d2,\p1) node[at start,left]{$1/a$};
\end{tikzpicture}
}
&
% (2) graph
\begin{tikzpicture}[scale=0.75]
\tikzmath{\s=1.25; \t=2.5; \d0=0.2; \d1=0.02; \d2=0.1; \p1=0.714; \p2=1.9;}
\coordinate (A) at (\d1,\d1);
\coordinate (B) at (\s-\d1,\d1);
\coordinate (C) at (\s-\d1,\p1);
\coordinate (D) at (\d1,\p2);
\draw[-,red,thick,fill] (A)--(B)--(C)--(D)--(A);
\placegrid
\draw[-,thin] (\s-\d2,\p1) -- (\s+\d2,\p1) node[at end,right]{$(1-s)/a$};
\draw[-,thin] (-\d2,\p2) -- (\d2,\p2) node[at start,left]{$1/a$};
\end{tikzpicture}
\\
\hline

\multicolumn{2}{|c|}{
% (3) condition
\rule{0pt}{10pt}
%$(3) \enspace a < 1 \;\land\; 1-a \leq s \;\land\; 1-at \leq s$
$(3) \quad 0 < 1-a \leq s \;\land\; 1-at \leq s$
}\\
% (3.1) condition
\rule{0pt}{10pt}
$(3.1) \quad t/a \leq 1$ &
% (3.2) condition
$(3.2) \quad t/a > 1$
\\[1ex]
% (3.1) graph
\begin{tikzpicture}[scale=0.75]
\tikzmath{\s=4; \t=2; \d0=0.2; \d1=0.02; \d2=0.1; \p1=0.94; \p2=0.9; \p3=2.9; \p4=1;}
\coordinate (A) at (\d1,\p1);
\coordinate (B) at (\d1,\t-\d1);
\coordinate (C) at (\p4,\t-\d1);
\coordinate (D) at (\p3,\d1);
\coordinate (E) at (\p2,\d1);
\draw[-,red,thick,fill] (A)--(B)--(C)--(D)--(E)--(A);
\placegrid
\draw[-,thin] (-\d2,\p1) -- (\d2,\p1) node[at start,left]{$t/a$};
\draw[-,thin] (\p2,-\d2) -- (\p2,\d2) node[at start,below]{$t-at$};
\draw[-,thin] (\p3,-\d2) -- (\p3,\d2) node[at start,below]{$1-at$};
\draw[-,thin] (\p4,\t-\d2) -- (\p4,\t+\d2) node[at end,above]{$1-a$};
\end{tikzpicture}
&
% (3.2) graph
\begin{tikzpicture}[scale=0.75]
\tikzmath{\s=5; \t=2; \d0=0.2; \d1=0.02; \d2=0.1; \p1=2.3; \p2=4.3; \p3=0.5; \p4=2.5;}
\coordinate (A) at (\p1,\d1);
\coordinate (B) at (\p3,\t-\d1);
\coordinate (C) at (\p4,\t-\d1);
\coordinate (D) at (\p2,\d1);
\draw[-,red,thick,fill] (A)--(B)--(C)--(D)--(A);
\placegrid
\draw[-,thin] (\p1,-\d2) -- (\p1,\d2) node[at start,below]{$t-at$};
\draw[-,thin] (\p2,-\d2) -- (\p2,\d2) node[at start,below]{$1-at\quad$};
\draw[-,thin] (\p3,\t-\d2) -- (\p3,\t+\d2) node[at end,above]{$t-a$};
\draw[-,thin] (\p4,\t-\d2) -- (\p4,\t+\d2) node[at end,above]{$1-a$};
\end{tikzpicture}
\\
\hline

\multicolumn{2}{|c|}{
% (4) condition
\rule{0pt}{10pt}
%$(4) \enspace a < 1 \;\land\; 1-a \leq s$
$(4) \quad 0 < 1-a \leq s \;\land\; (1-s)/a > t$
}\\[1ex]
% (4.1) condition
$(4.1) \quad t/a \leq 1$ &
% (4.2) condition
$(4.2) \quad t/a > 1$
\\[1ex]
% (4.1) graph
\begin{tikzpicture}[scale=0.75]
\tikzmath{\s=2.67; \t=2; \d0=0.2; \d1=0.02; \d2=0.1; \p1=0.52; \p2=1.26; \p3=1.33; \p4=1.13;}
\coordinate (A) at (\d1,\p2);
\coordinate (B) at (\d1,\t-\d1);
\coordinate (C) at (\p3,\t-\d1);
\coordinate (D) at (\s-\d1,\p1);
\coordinate (E) at (\s-\d1,\d1);
\coordinate (F) at (\p4,\d1);
\draw[-,red,thick,fill] (A)--(B)--(C)--(D)--(E)--(F)--(A);
\placegrid
\draw[-,thin] (-\d2,\p2) -- (\d2,\p2) node[at start,left]{$t/a$};
\draw[-,thin] (\s-\d2,\p1) -- (\s+\d2,\p1) node[at end,right]{$(1-s)/a$};
\draw[-,thin] (\p4,+\d2) -- (\p4,-\d2) node[at end,below]{$t-at$};
\draw[-,thin] (\p3,\t-\d2) -- (\p3,\t+\d2) node[at end,above]{$1-a$};
\end{tikzpicture}
&
% (4.2) graph
\begin{tikzpicture}[scale=0.75]
\tikzmath{\s=3.33; \t=2; \d0=0.2; \d1=0.02; \d2=0.1; \p1=0.5; \p2=1.02; \p3=2.5; \p4=2.2;}
\coordinate (A) at (\p1,\t-\d1);
\coordinate (B) at (\p3,\t-\d1);
\coordinate (C) at (\s-\d1,\p2);
\coordinate (D) at (\s-\d1,\d1);
\coordinate (E) at (\p4,\d1);
\draw[-,red,thick,fill] (A)--(B)--(C)--(D)--(E)--(A);
\placegrid
\draw[-,thin] (\s-\d2,\p2) -- (\s+\d2,\p2) node[at end,right]{$(1-s)/a$};
\draw[-,thin] (\p4,+\d2) -- (\p4,-\d2) node[at end,below]{$t-at$};
\draw[-,thin] (\p1,\t-\d2) -- (\p1,\t+\d2) node[at end,above]{$t-a$};
\draw[-,thin] (\p3,\t-\d2) -- (\p3,\t+\d2) node[at end,above]{$1-a$};
\end{tikzpicture}
\\
\hline

\multicolumn{2}{|c|}{
% (5) condition
\rule{0pt}{10pt}
$(5) \quad 1-a > s$
}\\
% (5.1) condition
\rule{0pt}{10pt}
$(5.1) \quad t/a \leq 1 \;\land\; t-at\leq s$ &
% (5.2) condition
$(5.2) \quad t/a \leq 1 \;\land\; (t-s)/a > t$
\\[1ex]

% (5.1) graph
\begin{tikzpicture}[scale=0.75]
%\tikzmath{\s=2.25; \t=3; \d0=0.2; \d1=0.02; \d2=0.1; \p1=1.8; \p2=2.12;}
\tikzmath{\s=1.875; \t=2.5; \d0=0.2; \d1=0.02; \d2=0.1; \p1=1.5; \p2=1.764;}
\coordinate (A) at (\d1,\p2);
\coordinate (B) at (\d1,\t-\d1);
\coordinate (C) at (\s-\d1,\t-\d1);
\coordinate (D) at (\s-\d1,\d1);
\coordinate (E) at (\p1,\d1);
\draw[-,red,thick,fill] (A)--(B)--(C)--(D)--(E)--(A);
\placegrid
\draw[-,thin] (-\d2,\p2) -- (\d2,\p2) node[at start,left]{$t/a$};
\draw[-,thin] (\p1,+\d2) -- (\p1,-\d2) node[at end,below]{$t-at\qquad$};
\end{tikzpicture}
&
% (5.2) graph
\begin{tikzpicture}[scale=0.75]
%\tikzmath{\s=1.8; \t=3; \d0=0.2; \d1=0.02; \d2=0.1; \p1=0.29; \p2=2.46;}
\tikzmath{\s=1.5; \t=2.5; \d0=0.2; \d1=0.02; \d2=0.1; \p1=0.24; \p2=2.05;}
\coordinate (A) at (\d1,\p2);
\coordinate (B) at (\d1,\t-\d1);
\coordinate (C) at (\s-\d1,\t-\d1);
\coordinate (D) at (\s-\d1,\p1);
\draw[-,red,thick,fill] (A)--(B)--(C)--(D)--(A);
\placegrid
\draw[-,thin] (-\d2,\p2) -- (\d2,\p2) node[at start,left]{\rule{0pt}{12pt}$t/a$};
\draw[-,thin] (\s-\d2,\p1) -- (\s+\d2,\p1) node[at end,right]{$(t-s)/a$};
\end{tikzpicture}
\\[1ex]
% (5.3) condition
%$(5.3) \enspace t/a > 1 \;\land\; t-a\leq s \;\land\; t-at \leq s$ &
$(5.3) \quad t/a > 1 \;\land\; t-at \leq s$ &
% (5.4) condition
%$(5.4) \enspace t/a > 1 \;\land\; t-a\leq s$
$(5.4) \quad t/a > 1 \;\land\; t < (t-s)/a < 1$
\\[1ex]
% (5.3) graph
\begin{tikzpicture}[scale=0.75]
\tikzmath{\s=2.67; \t=2; \d0=0.2; \d1=0.02; \d2=0.1; \p1=0.67; \p2=2.27;}
\coordinate (A) at (\p1,\t-\d1);
\coordinate (B) at (\s-\d1,\t-\d1);
\coordinate (C) at (\s-\d1,\d1);
\coordinate (D) at (\p2,\d1);
\draw[-,red,thick,fill] (A)--(B)--(C)--(D)--(A);
\placegrid
\draw[-,thin] (\p1,\t-\d2) -- (\p1,\t+\d2) node[at end, above]{$t-a$};
\draw[-,thin] (\p2,+\d2) -- (\p2,-\d2) node[at end,below]{$t-at\qquad$};
\end{tikzpicture}
&
% (5.4) graph
\begin{tikzpicture}[scale=0.75]
\tikzmath{\s=2.67; \t=2; \d0=0.2; \d1=0.02; \d2=0.1; \p1=0.54; \p2=1.6;}
\coordinate (A) at (\p2,\t-\d1);
\coordinate (B) at (\s-\d1,\t-\d1);
\coordinate (C) at (\s-\d1,\p1);
\draw[-,red,thick,fill] (A)--(B)--(C)--(A);
\placegrid
\draw[-,thin] (\p2,\t-\d2) -- (\p2,\t+\d2) node[at end, above]{$t-a$};
\draw[-,thin] (\s-\d2,\p1) -- (\s+\d2,\p1) node[at end,right]{$(t-s)/a$};
\end{tikzpicture}
\\
\hline

\end{tabular}
\caption{Case distinctions for polygon 133, showing all possibilities of
  admissible regions in the top left corner (depending on conditions for
  $a,s,t$). The empty cases (not shown) correspond to the conditions
  $1/a \leq t$ or $t-a \geq s$.}
\label{fig:cases133}
\end{figure}

We have seen that the different domains of definition for $A(s,t,a)$ are
polyhedra in~$\R^3$ (some of which are unbounded). In
Figures~\ref{fig:pwregions1} and~\ref{fig:pwregions2} two $2$-dimensional
slices of the set~\eqref{eq:prism} for particular choices of~$a$ are
shown. Note that in \fig{pwregions1} condition $C_5$ is not shown since it was
eliminated in the process of merging regions on which $A$ is defined by
the same expression. Moreover, $C_9\equiv a\leq1$ is not visible since its
plane $a=1$ is parallel to the depicted cross section $a=1.4$.

\begin{figure}
  \begin{center}
    \includegraphics[width=0.74\textwidth]{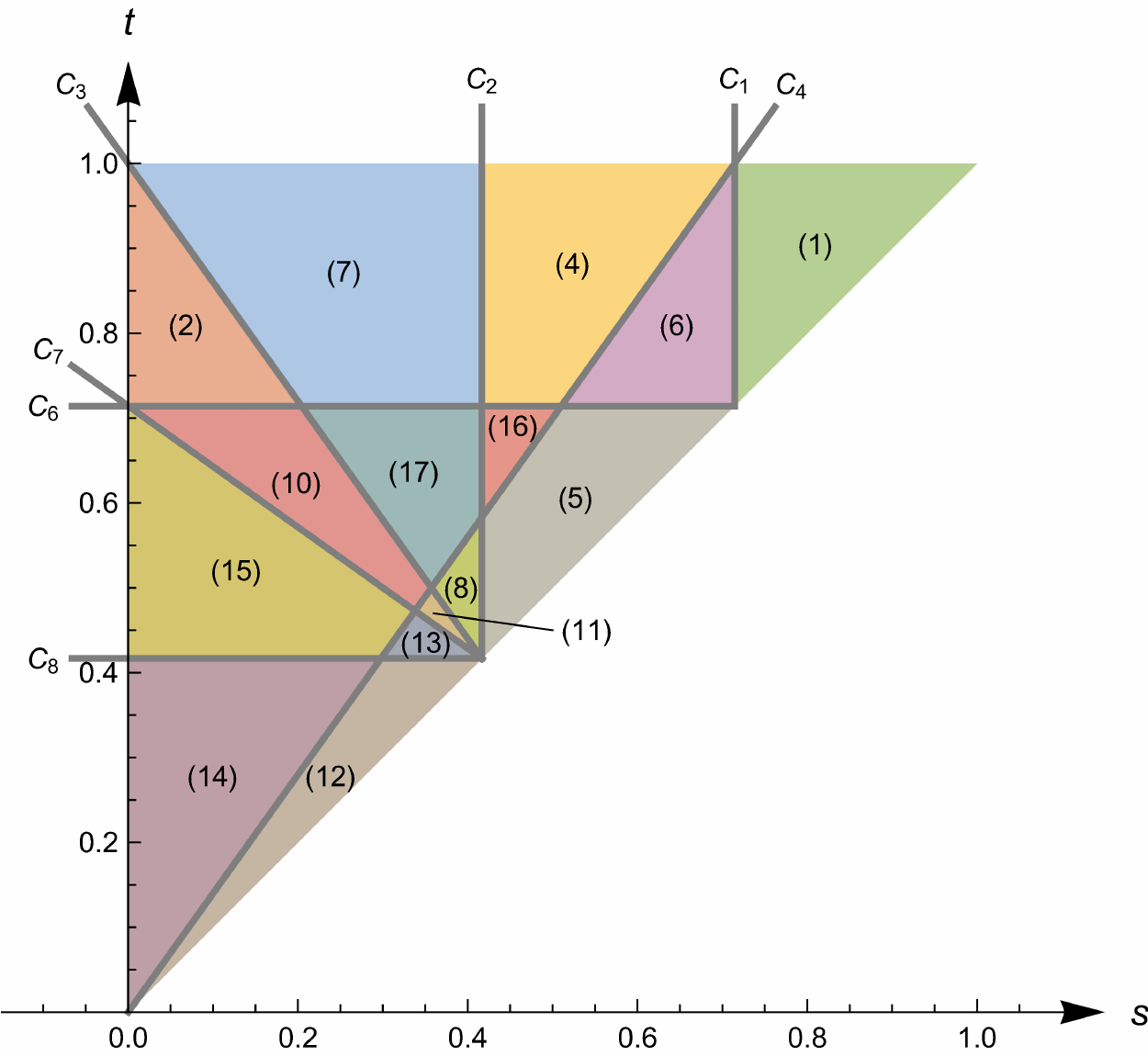}
  \end{center}
  \caption{Domains of definition of $A(s,t,a)$ for $a=1.4$, according to
    \lem{MGSTarea}.  Note that not all $17$ cases listed in the lemma are
    present for this particular choice of~$a$.}
  \label{fig:pwregions1}
  \vspace*{\floatsep}
  \begin{center}
    \includegraphics[width=0.74\textwidth]{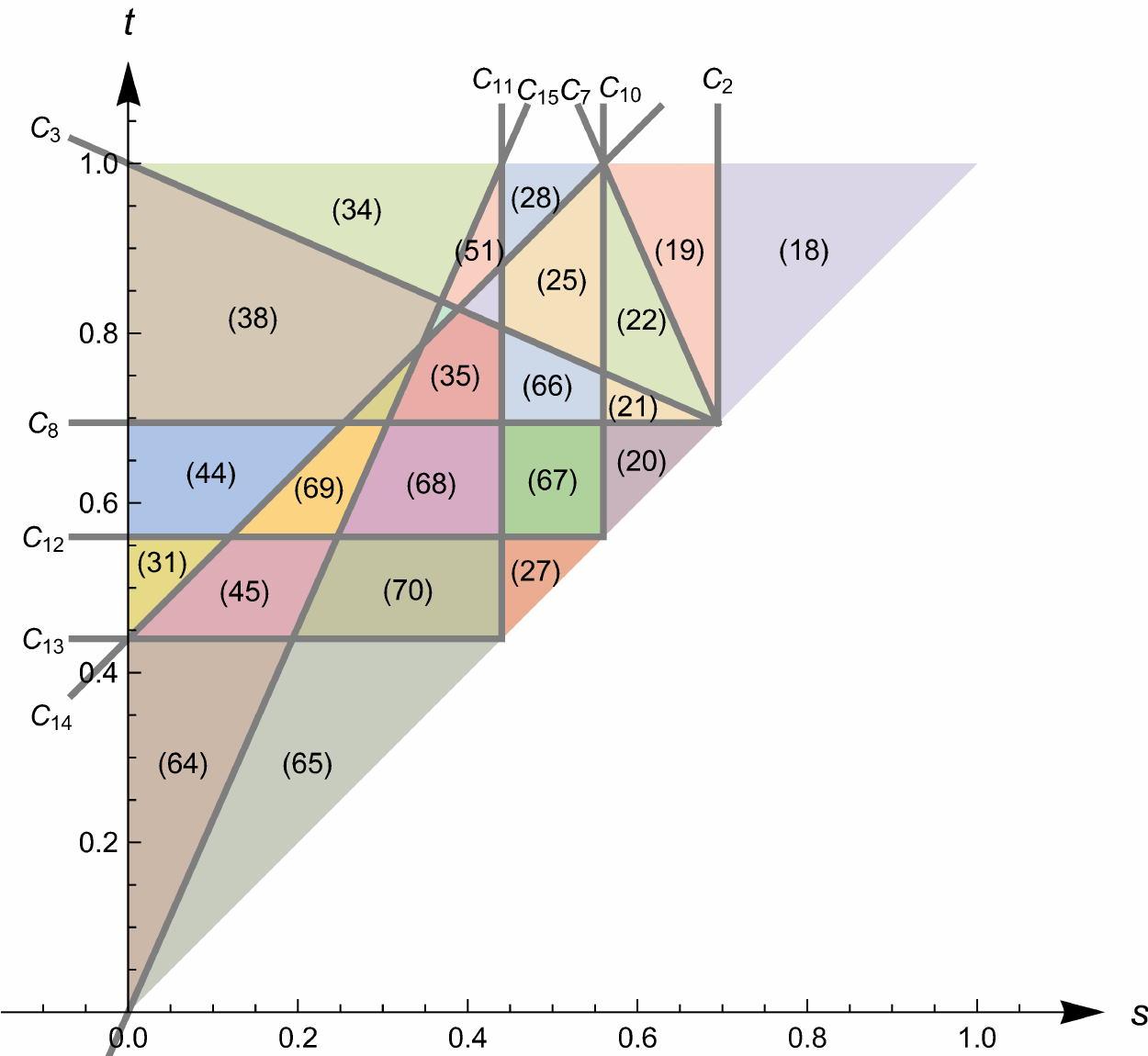}
  \end{center}
  \caption{Domains of definition of the area function $A(s,t,a)$ for $a=0.44$.}
  \label{fig:pwregions2}
\end{figure}

\pagebreak

\begin{lemma}\label{lem:min}
  For $a>0$, the minimum of the function $A(s,t,a)$ (defined in
  \lem{MGSTarea}) on the triangle $0\leq s\leq t\leq1$,
  \[
    m(a) := \min_{0\leq s\leq t\leq1} A(s,t,a)
  \]
  is given by a piecewise rational function, depending on~$a$, according to
  the following case distinctions (where we also give the location $(s_0,t_0)$
  of the minimum):
  \[
  \begin{array}{l@{\quad}c@{\quad}c@{\quad}c}
    & s_0 & t_0 & m(a) \\ \hline\rule{0pt}{16pt}
    0\leq a\leq \alpha_1 & \frac{(a-4) a}{a^3-a-4} & \frac{-2 a^2+4 a+2}{-a^3+a+4} & \frac{-a^4+2 a^3-2 a^2+6 a-4}{2 \left(a^3-a-4\right)} \\[1ex]
    \alpha_1\leq a\leq \alpha_2 & \frac{a \left(a^2-3\right)}{a^4-8 a-1} & \frac{a^3+a^2-5 a-1}{a^4-8 a-1} & \frac{a^3-2 a^2+a-2}{2 \left(a^4-8 a-1\right)} \\[1ex]
    \alpha_2\leq a\leq \alpha_3 & \frac{-2 a^3+2 a+1}{-a^4+8 a+3} & \frac{2 a^3+a^2-6 a-2}{a^4-8 a-3} & \frac{a^6+a^4-12 a^3+4 a^2-1}{2 a \left(a^4-8 a-3\right)} \\[1ex]
    \alpha_3\leq a\leq \alpha_4 & \frac{-2 a^2+a+1}{-4 a^3+5 a^2+6 a+1} & \frac{-2 a^3+a^2+4 a+1}{-4 a^3+5 a^2+6 a+1} & \frac{4 a^4-9 a^3+2 a^2+a-2}{2 \left(4 a^3-5 a^2-6 a-1\right)} \\[1ex]
    \alpha_4\leq a\leq \alpha_5 & \frac{a^3+a+1}{-4 a^3+3 a^2+6 a+1} & \frac{2 a^2+4 a+1}{-4 a^3+3 a^2+6 a+1} & \frac{4 a^4-4 a^3+a-2}{2 \left(4 a^3-3 a^2-6 a-1\right)} \\[1ex]
    \alpha_5\leq a\leq \alpha_6 & -\frac{3 a^2+a-1}{4 a^3-4 a^2-4 a+1} & \frac{-4 a^2-2 a+1}{4 a^3-4 a^2-4 a+1} & \frac{8 a^3-4 a^2-5 a+2}{2 \left(4 a^3-4 a^2-4 a+1\right)} \\[1ex]
    \alpha_6\leq a\leq \alpha_7 & \frac{2 a+1}{7 a+1} & \frac{8 a^2+6 a+1}{7 a^2+8 a+1} & \frac{-2 a^2+3 a+2}{2 (a+1) (7 a+1)} \\[1ex]
    \alpha_7\leq a\leq 1 & \frac{(a+1)^2}{a (7 a+4)} & \frac{(a+1) (4 a+1)}{a (7 a+4)} & \frac{-7 a^4+6 a^3+6 a^2-2 a-1}{2 a^2 (7 a+4)} \\[1ex]
    1\leq a\leq \alpha_8 & \frac{(a+1)^2}{a^4+2 a^3+3 a^2+2 a+3} & \frac{(a+1) \left(a^2+2 a+2\right)}{a^4+2 a^3+3 a^2+2 a+3} & \frac{a^4-a^2-2 a+4}{2 a \left(a^4+2 a^3+3 a^2+2 a+3\right)} \\[1ex]
    \alpha_8\leq a  & \frac{a+1}{a^2+2 a+3} & \frac{a^2+2 a+2}{a^2+2 a+3} & \frac{1}{2 a \left(a^2+2 a+3\right)}
  \end{array}
  \]
  Here, the quantities $\alpha_1,\dots,\alpha_8$ stand for the following
  algebraic numbers, where $\operatorname{Root}(p,I)$ denotes the unique
  real root of the polynomial~$p$ in the interval~$I$:
  \begin{alignat*}{2}
    \alpha_1 &= 0.295597... &&=\operatorname{Root}\bigl(a^3+a^2+3 a-1,[0,1]\bigr), \\
    \alpha_2 &= 0.395065... &&=\operatorname{Root}\bigl(a^5-9 a^2+a+1,[0,1]\bigr), \\
    \alpha_3 &= 0.405669... &&=\operatorname{Root}\bigl(2 a^4-a^3-6 a^2+1,[0,1]\bigr), \\
    \alpha_4 &= 0.553409... &&=\operatorname{Root}\bigl(12 a^4-15 a^3-24 a^2+5 a+6,[0,1]\bigr), \\
    \alpha_5 &= 0.622179... &&=\operatorname{Root}\bigl(4 a^3-8 a^2-3 a+4,[0,1]\bigr), \\
    \alpha_6 &= 0.647363... &&=\operatorname{Root}\bigl(8 a^2+a-4,[0,1]\bigr) = \tfrac{1}{16}\bigl(\sqrt{129}-1\bigr), \\
    \alpha_7 &= 0.931478... &&=\operatorname{Root}\bigl(7 a^3-5 a-1,[0,1]\bigr), \\
    \alpha_8 &= 1.174559... &&=\operatorname{Root}\bigl(a^3+a^2-3,[1,2]\bigr).
  \end{alignat*}
\end{lemma}
\begin{proof}
  We locate the minimum in a similar fashion as in \sec{schur}, by identifying
  points $(s,t)$ where the gradient of the area function~$A$ vanishes. What
  complicates our task is the additional parameter~$a$.  Since $A$ is defined
  in pieces, it may not be differentiable at the boundaries between different
  regions, and therefore, we should be aware that such locations could contain
  the minimum. For each region~$(R_i)$, $1\leq i\leq70$, on which $A(s,t,a)$
  is defined, we perform the following steps:
  
  \begin{itemize}
    \setlength{\itemsep}{1ex}
  \item compute the gradient $\left(\frac{\partial A}{\partial s},
    \frac{\partial A}{\partial t}\right)$,
  \item find all points $(s,t)$ where the gradient is zero, and
  \item for each such point determine for which values of~$a$ it
    actually lies in $(R_i)$.
  \end{itemize} 

  On the region $(R_1)$ from \lem{MGSTarea}, the gradient
  of~$A$ is $\frac1a(s-t+1,t-s-1)$, which vanishes on all points $(s,s+1)$;
  however, since the region $(R_1)$ is characterized by $\overline{C_1}\equiv s>\frac1a$
  (and the general condition $s\leq t\leq1$), one sees that none of these
  points lie in it. Continuing in this manner, we find that in each of the
  regions~$(R_2)-(R_{70})$ there is exactly one point $(s,t)$ for which the
  gradient of~$A$ vanishes, but in most cases this point lies outside the
  region for all~$a$.  For example, on $(R_2)$ the gradient is
  $\frac1a(2as-2at+2s-t+a,t-2as-s)$, which equals zero for
  \begin{equation}\label{eq:locgradvan}
    (s,t) = \left(\frac{a}{4a^2+2a-1},\frac{a(2a+1)}{4a^2+2a-1}\right).
  \end{equation}
  In order to find the values of~$a$ that give us that $(s,t)\in (R_2)$, the
  conditions defining $(R_2)$ (plus the global assumptions) need to be satisfied,
  namely:
  \[
    a s+t\leq 1 \;\land\; t\geq a s \;\land\; at>1 \;\land\; 0<s<t<1.
  \]
  After substituting $s$ and $t$ with the right hand side of (\ref{eq:locgradvan}) and clearing
  denominators, one gets a system of polynomial inequalities, involving only
  the variable~$a$.  Cylindrical algebraic decomposition simplifies it to
  \[
    a \geq \operatorname{Root}\bigl(2a^3-3a^2-2a+1,[1,2]\bigr)
    = 1.889228559...
  \]
  Hence, for each $a$ satisfying this condition we have a local minimum
  at the point given in~\eqref{eq:locgradvan}.

  We proceed  in similar fashion and identify $17$ local minima, each
  occurring only for $a$ in a certain interval. Some of these intervals
  partly overlap, which means that we have to study a subdivision of the
  positive real line that is a refinement of all $17$ intervals.  When two
  functions intersect in the interior of an interval, it is split into
  two subintervals. CAD is once again employed to find the smallest among
  the local minima; this is done individually for each of the refined
  intervals. As a result, we obtain the piecewise description of the
  function~$m(a)$ given above; see \fig{minplot} and the accompanying
  electronic material~\cite{KoutschanWong19} for details. 
  
  \begin{figure}[t]
    \begin{center}
      \includegraphics[width=0.84\textwidth]{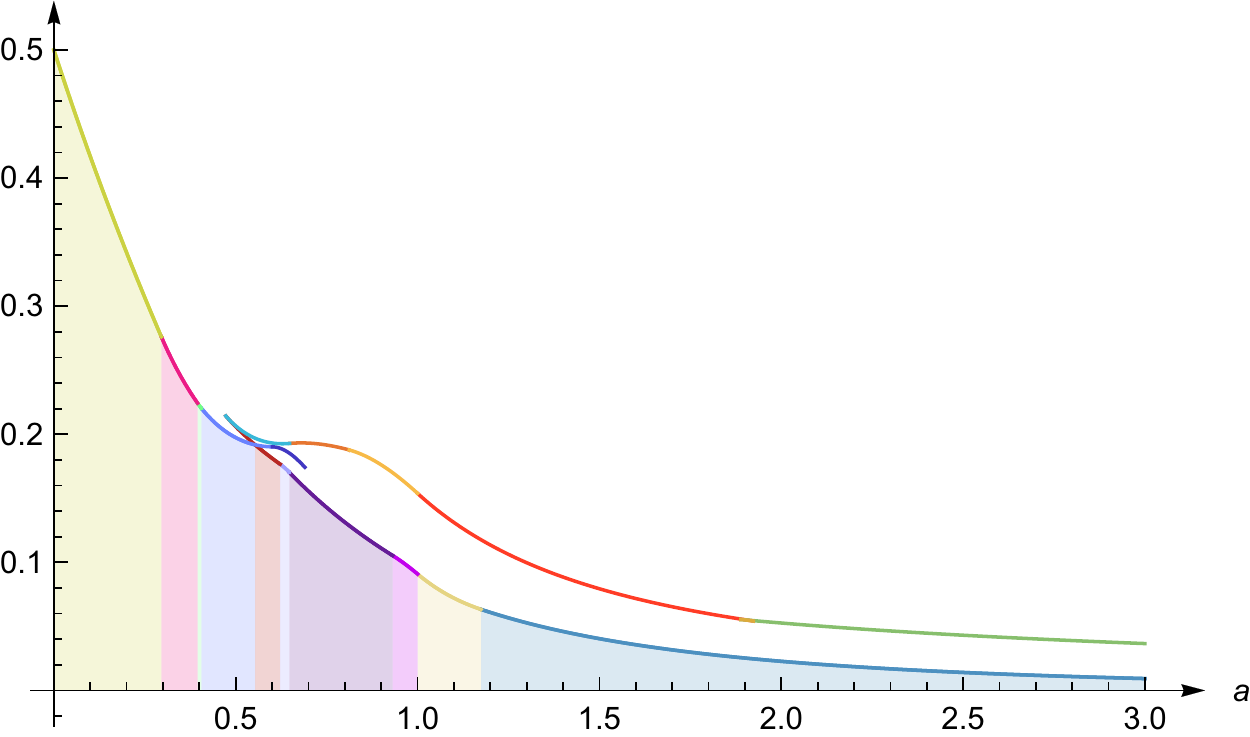}
    \end{center}
    \caption{Plot of $A(s,t,a)$ on the $17$ different intervals of~$a$
      identified from the $17$ local minima in the proof of \lem{min} for
      $0\leq a\leq3$; the shading under the graph indicates the main $10$
      intervals that are needed to describe the global minimum
      function~$m(a)$.}
    \label{fig:minplot}
  \end{figure}

  It is clear from construction that $A(s,t,a)$ must be a continuous function,
  since the admissible polygons (shaded regions in \fig{MGSTarea}) cannot jump
  or disappear if the parameters $a,s,t$ are changed infinitesimally, i.e., if
  the lines in \fig{MGSTarea} are shifted or slanted by a little bit.  In
  contrast, it is not obvious why it should be differentiable.
  Therefore, there is a possibility
  that the minimum can occur where the derivative does not exist. Hence, it
  is necessary to study the values of $A(s,t,a)$ along the boundaries of the
  different domains of definition. To accomplish this task, we view $A$ as a
  bivariate function in $s$ and~$t$, with a parameter~$a$. For each inequality in
  the list of conditions~\eqref{eq:pwconds}, the corresponding equation defines
  a line in~$\R^2$. For each such line, we proceed to determine the range
  of~$a$ for which the line intersects the triangle $0\leq s\leq t\leq1$. On
  the resulting line segment, the pieces of $A(s,t,a)$ are given by univariate
  polynomials, still involving the parameter~$a$. Equating their derivatives to
  zero, we find all of the local minima on this line segment, which could give rise to
  local minima of $A(s,t,a)$. After looking at all $16$ lines, each of which
  splits into at most $70$ segments, we find $225$ candidates for minima.
  CAD confirms that none of them are actually smaller than the one given
  by $m(a)$. This fact also becomes apparent by plotting these candidates
  against the function~$m(a)$, as shown in \fig{minloc12} (top part).

  Finally, we should also check all points where any two lines defined
  by~\eqref{eq:pwconds} intersect. We find $54$ points that lie inside the
  triangle $0\leq s\leq t\leq1$, at least for certain
  choices of~$a$. The value of $A(s,t,a)$ at a particular point is given by a
  piecewise function depending on~$a$.  Assembling all pieces for all points,
  we obtain $348$ cases. For each of them, CAD confirms (rigorously!) that the
  value of $A(s,t,a)$ does not go below~$m(a)$. A ``non-rigorous proof'' of this
  fact is shown in \fig{minloc12} (bottom part).

  Summarizing, we have shown that, for each particular choice of $a>0$, the
  minimum of the function $A(s,t,a)$ on the triangle $0\leq s\leq t\leq1$ is
  given by~$m(a)$, and we have determined the location $(s_0,t_0)$ where this
  minimum is attained. This immediately establishes an asymptotic lower bound
  for MGSTs on$[n]$, as $n$ goes to infinity.
\end{proof}

\begin{figure}[t]
  \begin{center}
    \includegraphics[width=0.77\textwidth]{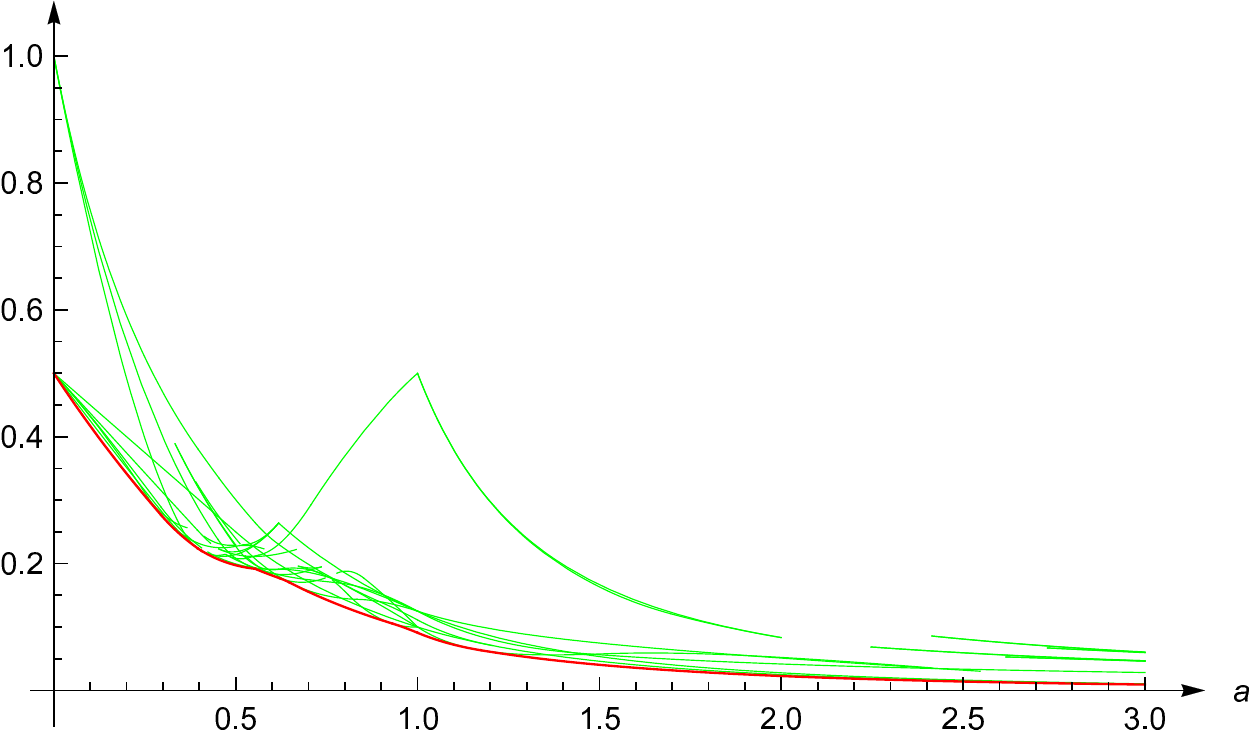} \\[2ex]
    \includegraphics[width=0.77\textwidth]{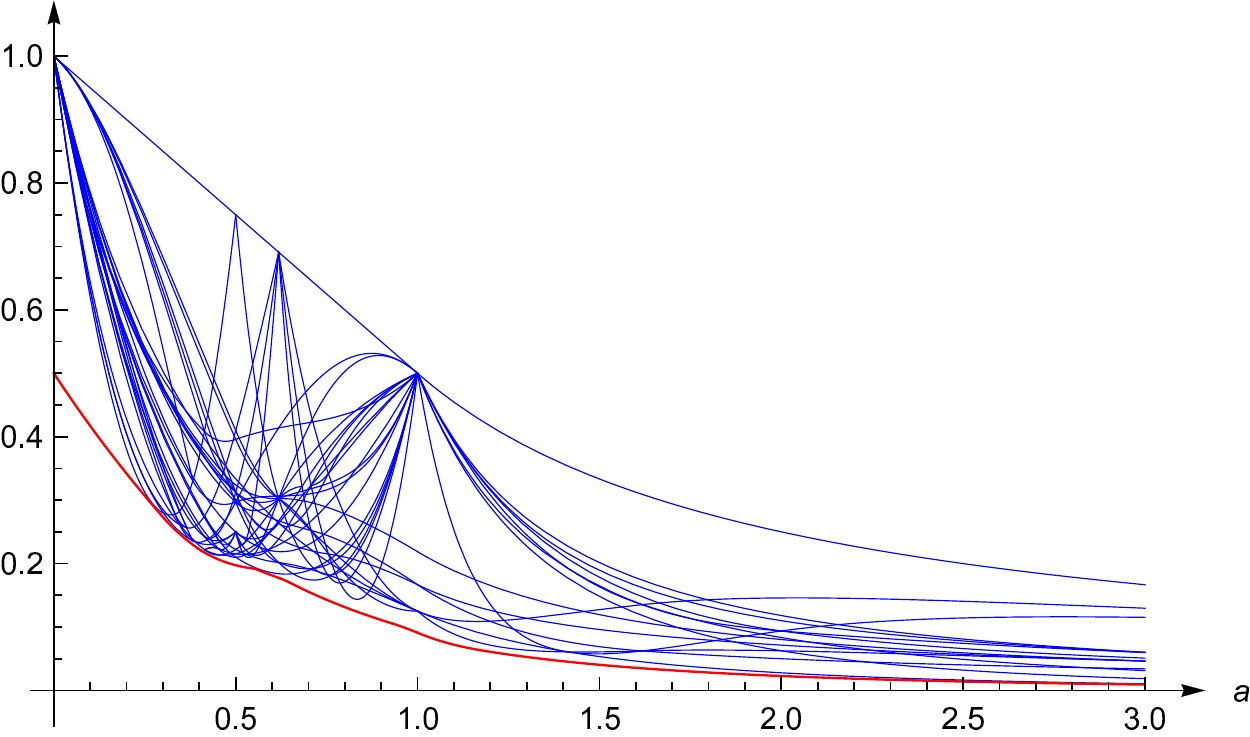}
  \end{center}
  \caption{Global minimum of $A(s,t,a)$ (red curve) compared to potential
    minima along lines (green curves, top part) and potential minima on intersection points
    (blue curves, bottom part).}
  \label{fig:minloc12}
\end{figure}

We wrap up this section with some remarks on the consquences of \lem{min} and
on what appears to be erratic (jumpy) behavior for some values of $a$ in
\fig{minloc12}. We assure the reader that it is not due to the amount of
alcohol that was consumed throughout this meal, but rather an indication of
the appearance and disappearance of certain admissible regions for the MGSTs
as $a$ changes.

First, we would like to note that \lem{min} explains why the asymptotic
formula for MGSTs for integral $a\geq2$ given
in~\cite{ButlerCostelloGraham10,Thanatipanonda09,ThanatipanondaWong17} does
not specialize to the previously known case $a=1$: this phenomenon is due to
the piecewise definition of~$m(a)$, with a transition at $1<\alpha_8<2$.
Geometrically speaking, $\alpha_8$ marks the point where the polygon 133
(see \fig{MGSTarea}) disappears, when $a$ increases from $1$ to $2$, and
$s=s_0(a)$ and $t=t_0(a)$ are updated constantly. 

A second interesting finding that follows from \lem{min} is that there is a
jump of $\bigl(s_0(a),t_0(a)\bigr)$ at $a=\alpha_4=0.5534...$; the function
$m(a)$ however is continuous. In \fig{minplot} one sees that at $a=\alpha_4$
the functions of two local minima intersect, and therefore this point marks
the jump from one branch to another one. In \fig{MGSTarea3} the situation is
shown for two different values of $a$ close to $\alpha_4$: while the shaded area
in both parts of the figure is almost the same, the values of $s$ and $t$
change quite dramatically. We invite the reader to play with such transitions
in the accompanying electronic material~\cite{KoutschanWong19}.

In the next section, we bring up the fact that the coloring pattern of three blocks that we
generously assumed for $a>0$ does not actually give the global minimum
on $0<a<1$ over any 2-coloring of $[n]$ and we take care to emphasize this in
the statement of the theorems. This will therefore explain the erratic
behavior at $a=1$ in both graphs of \fig{minloc12}.

%%%%%%%%%%%%%%%%%%%%%%%%%%%%%%%%%%%%%%%%%%%%%%%%%%%%%%%%%%%%%%%%%%%%%%%%%%%%%%%%%%%%%%%%%%%%%%%
\section{Exact bounds for generalized Schur triples}
\label{sec:discrete}
%%%%%%%%%%%%%%%%%%%%%%%%%%%%%%%%%%%%%%%%%%%%%%%%%%%%%%%%%%%%%%%%%%%%%%%%%%%%%%%%%%%%%%%%%%%%%%%

In this section we apply the results from the last section, i.e., from the
continuous setting, to the discrete enumeration problem of monochromatic
generalized Schur triples (MGSTs). Hence, $s$ and $t$ are now integers with
$1\leq s\leq t\leq n$ that describe the coloring $R^sB^{t-s}R^{n-t}$ of~$[n]$.
Throughout this section we use the convention that a sum whose lower bound is
greater than its upper bound is zero, i.e.,
\[
  \sum_{x=i}^j f(x) =
  \begin{cases}
    f(i) + \dots + f(j), & \text{if } i\leq j, \\
    0, & \text{if } i>j.
  \end{cases}
\]

Analogous to \sec{schur} we use the notation $\M^{(a)}$ to count MGSTs. More
precisely, we define $\M^{(a)}(n,s,t)$ and $\M^{(a)}(n)$, as follows:
\begin{align*}
  \M^{(a)}(n,s,t) &\colonequal \bigl|\bigl\{ T=(x,y,x+\lfloor ay\rfloor)\in[n]^3 : \\
  & \qquad\quad T \in ([s]\cup\{t+1,\dots,n\})^3 \lor T\in\{s+1,\dots,t\}^3 \bigr\}\bigr|, \\[1ex]
  \M^{(a)}(n) &\colonequal \min_{1\leq s\leq t\leq n} \M^{(a)}(n,s,t).
\end{align*}
In contrast to the previous section, we will now mostly look at special cases
for~$a$, since we cannot hope to get an exact formula for the minimal number
of MGSTs for general~$a\in\R^+$.

\begin{lemma}\label{lem:MGSTcount}
  Let $a\in\R$ with $a\geq1$ and let $n,s,t\in\N$ with $1\leq s\leq t\leq n$.
  Furthermore, assume that the inequalities $as+t\geq n$, $t\geq as$, and
  $s+as\leq t$\linebreak[4] hold. Then the number $\M^{(a)}(n,s,t)$ of monochromatic
  generalized Schur triples of $[n]$ under the coloring $R^sB^{t-s}R^{n-t}$
  is given by
  %% \[
  %% \sum_{y=1}^{\lceil s/a\rceil} \; \sum_{x=1}^{s-\lfloor a y\rfloor} \!\! 1 \; + \!\!
  %% \sum_{y=s+1}^{\lceil (t-s-1)/a\rceil} \; \sum_{x=s+1}^{t-\lfloor a y\rfloor} \!\! 1 \; + \!\!
  %% \sum_{y=1}^{\lceil (n-t)/a\rceil} \; \sum_{x=t+1}^{n-\lfloor a y\rfloor} \!\! 1 \; + \!\!
  %% \sum_{y=t+1}^{\lceil (n-1)/a\rceil} \; \sum_{x=1}^{n-\lfloor a y\rfloor } \!\! 1.
  %% \]
  \[
  \sum_{y=1}^{\lfloor s/a\rfloor} \; \sum_{x=1}^{s-\lfloor a y\rfloor} \!\! 1 \;\; +
  \sum_{y=s+1}^{\lfloor (t-s)/a\rfloor} \; \sum_{x=s+1}^{t-\lfloor a y\rfloor} \!\! 1 \;\; +
  \sum_{y=1}^{\lfloor (n-t)/a\rfloor} \; \sum_{x=t+1}^{n-\lfloor a y\rfloor} \!\! 1 \;\; + \;
  \sum_{y=t+1}^{\lfloor n/a\rfloor} \; \sum_{x=1}^{n-\lfloor a y\rfloor } \!\! 1.
  \]
  Moreover, the explicit list of these MGSTs $(x,y,x+\lfloor ay\rfloor)$
  can be directly read off from the above formula.
\end{lemma}
\begin{proof}
  Under the given assumptions, we have to consider monochromatic triples of
  types $111$, $222$, $313$, and $133$, see, e.g., \fig{MGSTarea}. Obviously,
  the four sums correspond exactly to these four cases. Note that if $at>n$,
  then the case $133$ is not present, which is reflected by the fact that the
  corresponding sum is zero in this case.
\end{proof}

The assumed inequalities in \lem{MGSTcount} tell us that we are either in
$(R_7)$ (when $at>n$) or in $(R_{17})$ (when $at\leq n$); these regions were
introduced in \lem{MGSTarea}. Recall $\alpha_8=1.174559...$ from \lem{min},
and also that the global minimum of the area function $A(s,t,a)$ is located in
$(R_7)$ (when $a\geq\alpha_8$) or in $(R_{17})$ (when $1\leq a\leq\alpha_8$).

\pagebreak

\begin{theorem}\label{thm:MGSTcount2}
  The minimal number of monochromatic generalized Schur triples of the form
  $(x,y,x+2y)$ that can be attained under any $2$-coloring of $[n]$ of the
  form $R^sB^{t-s}R^{n-t}$ is
  \[
    \M^{(2)}(n) = \left\lfloor\frac{n^2 - 10 n + 33}{44}\right\rfloor.
  \]
\end{theorem}
\begin{proof}
  For $a=2$ we clearly have $\alpha_8\leq a$, and by \lem{min} it follows
  that the optimal choice for $s$ and~$t$ is expected around the point
  \[
    n\cdot\left(\frac{a+1}{a^2+2 a+3}, \frac{a^2+2 a+2}{a^2+2 a+3}\right) =
    \left(\frac{3n}{11},\frac{10n}{11}\right).
  \]
  The three conditions $2s+t\geq n$, $t\geq 2s$, $3s\leq t$ are satisfied (at
  least for large~$n$), and therefore we can use \lem{MGSTcount} to compute
  the exact number of MGSTs:
  \begin{multline*}
    \M^{(2)}(n,s,t) =
    \sum_{y=1}^{\lfloor s/2\rfloor} \sum_{x=1}^{s-2y} 1 \;\; + \!\!
    \sum_{y=s+1}^{\lfloor (t-s-1)/2\rfloor} \sum_{x=s+1}^{t-2y} \! 1 \;\; + \!\!
    \sum_{y=1}^{\lfloor (n-t)/2 \rfloor} \sum_{x=t+1}^{n-2 y} \! 1 = {} \\ =
    \biggl\lfloor \frac{s}{2}\biggr\rfloor\hspace{-0.471pt}\left\lfloor \frac{s-1}{2}\right\rfloor +
    \left\lfloor \frac{n-t}{2}\right\rfloor\hspace{-0.471pt} \hspace{-0.471pt}\left\lfloor \frac{n-t-1}{2}\right\rfloor +
    \left\lfloor \frac{t-s}{2}\right\rfloor \left\lfloor \frac{t-s-1}{2}\right\rfloor +
    2 s^2-s t+s.
  \end{multline*}
  From now on, we proceed in an analogous fashion as in the proofs of
  \lem{minMnst} and \thm{minMn}.  Empirically, we find that for each $n\in\N$,
  the minimum of $\M^{(2)}(n,s,t)$ is attained at
  \[
    s_0 = \left\lfloor\frac{3n+1}{11}\right\rfloor, \qquad
    t_0 = \left\lfloor\frac{10n}{11}\right\rfloor +
    \begin{cases} -1, & \text{if } n=22k+10, \\ 0, & \text{otherwise}.\end{cases}
  \]
  When we plug in $s_0+i$ and $t_0+j$ into the above formula for
  $\M^{(2)}(n,s,t)$, we need to make a case distinction $n=22k+\ell$ for
  $0\leq\ell\leq21$ in order to get rid of the floors. Moreover, we need to
  distinguish even and odd $i$ (resp.~$j$). Evaluating and simplifying
  \[
    \M^{(2)}(22k+\ell, s_0+2i_1+i_2, t_0+2j_1+j_2), \quad 0\leq\ell\leq21, \; i_2,j_2\in\{0,1\},
  \]
  we obtain $88$ polynomials in $i_1,j_1,k$. Applying CAD individually to each of
  these polynomials and by checking a few values explicitly (not unlike what we
  did in the proof of \lem{MGSTcount}), one proves that the minimum is indeed
  attained at $(s_0,t_0)$. Finally, one evaluates $\M^{(2)}(22k+\ell,s_0,t_0)$ for
  all $\ell=0,\dots,21$ and finds that it is always of the form
  $\frac{1}{44}\bigl(n^2-10n\bigr)+\delta_{\ell}$, where the values $\delta_0,\dots,\delta_{21}$ are
  \[
    0,\tfrac{9}{44},\tfrac{4}{11},\tfrac{21}{44},\tfrac{6}{11},\tfrac{25}{44},\tfrac{6}{11},\tfrac{21}{44},
    \tfrac{4}{11},\tfrac{9}{44},0,\tfrac{3}{4},\tfrac{5}{11},\tfrac{5}{44},\tfrac{8}{11},\tfrac{13}{44},
    -\tfrac{2}{11},\tfrac{13}{44},\tfrac{8}{11},\tfrac{5}{44},\tfrac{5}{11},\tfrac{3}{4}.
%    0, 9, 16, 21, 24, 25, 24, 21, 16, 9, 0, 33, 20, 5, 32, 13, -8, 13, 32, 5, 20, 33.
  \]
  Since the largest value is $\frac34$ and since the smallest value is greater
  than $-\frac14$ (i.e., all values $\delta_{\ell}$ lie inside an interval of
  length~$1$), the claimed formula follows.

  One last detail: we still have to examine for which~$n$ the conditions
  $2s+t\geq n$, $t\geq 2s$, $3s\leq t$ are satisfied, as it could happen that
  for small~$n$ the point $(s_0,t_0)$ lies not inside the correct region
  $(R_{17})$, due to the rounding errors. With the (somewhat generous)
  assumptions $\frac{3n+1}{11}-1\leq s\leq\frac{3n+1}{11}$ and
  $\frac{10n}{11}-2\leq t\leq\frac{10n}{11}$ we find that the above conditions
  are satisfied for all $n\geq25$. For the remaining values $n<25$, the
  claimed formula can be verified by an explicit computation.
\end{proof}

\begin{theorem}\label{thm:MGSTcount3}
  The minimal number of monochromatic generalized Schur triples of the form
  $(x,y,x+3y)$ that can be attained under any $2$-coloring of $[n]$ of the
  form $R^sB^{t-s}R^{n-t}$ is
  \[
    \M^{(3)}(n) = \left\lfloor\frac{n^2 - 18 n + 101}{108}\right\rfloor +
    \begin{cases}
      1, & \text{if } n=54k+36, \\
      -1, & \text{if } n=54k+30 \text{ or } n=54k+42, \\
      0, & \text{otherwise}.
    \end{cases}
  \]
\end{theorem}
\begin{proof}
  For $a=3$, it follows by \lem{min} that the optimal choice for $s$ and~$t$
  is expected around the point
  \[
    n\cdot\left(\frac{a+1}{a^2+2 a+3}, \frac{a^2+2 a+2}{a^2+2 a+3}\right) =
    \left(\frac{4n}{18},\frac{17n}{18}\right).
  \]
  This means that the proof will require $18\cdot a=54$ case distinctions
  $n=54k+\ell$ for $0\leq\ell\leq53$.  Empirically, we find that for each
  $n\in\N$, the minimum of $\M^{(3)}(n,s,t)$ is attained at
  \begin{align*}  
    s_0 &= \left\lfloor\frac{4n}{18}\right\rfloor -
    \begin{cases}
      1, & \text{if } n=54k+18, \\
      0, & \text{otherwise},
    \end{cases} \\
    t_0 &= \left\lfloor\frac{17n}{18}\right\rfloor -
    \begin{cases}
      1, & \text{if } n=9k+i \text{ for } i\in\{3,4,7,8\}, \\
      2, & \text{if } n=54k+18, \\
      0, & \text{otherwise}.
    \end{cases}
  \end{align*}
  Applying CAD to the $486$ polynomials
  \[
    \M^{(3)}(54k+\ell,s_0+3i_1+i_2,t_0+3j_1+j_2), \quad 0\leq\ell\leq53,\; i_2,j_2\in\{0,1,2\},
  \]
  proves that our choice of $(s_0,t_0)$ locates the
  minimum. Evaluating \\ $\M^{(3)}(n,s_0,t_0)$ for $n=54k+\ell$, one obtains
  $\frac{1}{108}\bigl(n^2-18n\bigr)+\delta_{\ell}$, where $\delta_{36}=1$,
  $\delta_{30}=\delta_{42}=-\frac13$, and all remaining $\delta_{\ell}$ range
  from $-\frac1{27}$ to $\frac{101}{108}$. Hence, the claimed formula follows.
\end{proof}

\begin{theorem}\label{thm:MGSTcount4}
  The minimal number of monochromatic generalized Schur triples of the form
  $(x,y,x+4y)$ that can be attained under any $2$-coloring of $[n]$ of the
  form $R^sB^{t-s}R^{n-t}$ is
  \[
    \M^{(4)}(n) = \left\lfloor\frac{n^2 - 28 n + 245}{216}\right\rfloor -
    \begin{cases}
      1, & \text{if } n=108k+i \mbox{ for } i \in I,\\
      0, & \text{otherwise},
    \end{cases}
  \]
  where $I=\lbrace 0, 1, 27, 28, 43, 47, 48, 53, 58, 63, 67, 68, 69, 73, 78, 
   83, 88, 89, 93
   \rbrace$.
\end{theorem}
\begin{proof}
  For $a=4$, it follows by \lem{min} that the optimal choice for $s$ and~$t$
  is expected around the point
  \[
    n\cdot\left(\frac{a+1}{a^2+2 a+3}, \frac{a^2+2 a+2}{a^2+2 a+3}\right) =
    \left(\frac{5n}{27},\frac{26n}{27}\right).
  \]
  This means that the proof will require $27\cdot a=108$ case distinctions
  of the form $n=108k+\ell$ for $0\leq\ell\leq 107$.  Empirically, we find
  that for each $n\in\N$, the minimum of $\M^{(4)}(n,s,t)$ is attained at
  \begin{align*}  
    s_0 &= \left\lfloor\frac{5n-4}{27}\right\rfloor +
    \begin{cases}
      -1, & \text{if } n=108k+28, \\
      1, & \text{if } n=108k+i \text{ for } i\in\{0, 87, 103\}, \\
      0, & \text{otherwise}.
    \end{cases} \\
    t_0 &= \left\lfloor\frac{26n-34}{27}\right\rfloor +
    \begin{cases}
      -1, & \text{if } n=108k+i \text{ for } i\in\{28, 33, 38, 43\}, \\
      1, & \text{if } n=108k+i \\
      & \text{ for } i\in\{1, 77, 78, 82, 83, 88, 93, 98, 104\},\\
     2, & \text{if } n=108k+i \text{ for } i\in\{0, 87, 103\}, \\
      0, & \text{otherwise}.
    \end{cases}
  \end{align*}
  Applying CAD to the $1728$ polynomials
  \[
    \M^{(4)}(108k+\ell,s_0+4i_1+i_2,t_0+4j_1+j_2), \quad 0\leq\ell\leq107,\; i_2,j_2\in\{0,1,2,3\},
  \]
  proves that our choice of $(s_0,t_0)$ locates the
  minimum. Evaluating\\ $\M^{(4)}(n,s_0,t_0)$ for $n=108k+\ell$,
  $0\leq\ell\leq107$, one obtains 108 polynomials of the form
  $\frac{1}{216}\bigl(n^2-28n\bigr)+\delta_{\ell}$.  At this point, the
  analysis deviates a bit from the previous two theorems, because we observe
  that the range of the computed $\delta_{\ell}$'s is much larger
  than~$1$. Therefore, we would like to choose an appropriate interval to
  contain the largest number of $\delta_{\ell}$ such that we minimize the
  number of exceptional cases (i.e., the necessary corrections resulting from
  applying the floor function to numbers that are out of range).

  To accomplish this, we find that shifting all of the values down by
  $\tfrac{29}{216}$ gives the minimum number (19, to be precise) of
  $\delta_{\ell}$ that are not within range (i.e., not in $[0,1)$). We now
  realize that these are the values that give us our desired count, so we add
  1 to make sure it is recognized by the floor function. Hence, the optimal
  delta is $\tfrac{29}{216}+1=\frac{245}{216}.$ Finally, for each
  of the nineteen $\delta_{\ell}$'s that are out of bounds (in this case, less
  than~$0$), we remove $1$ and this gives us our claimed formula.
\end{proof}

\begin{figure}[t]
  \begin{center}
    \includegraphics[width=0.4\textwidth]{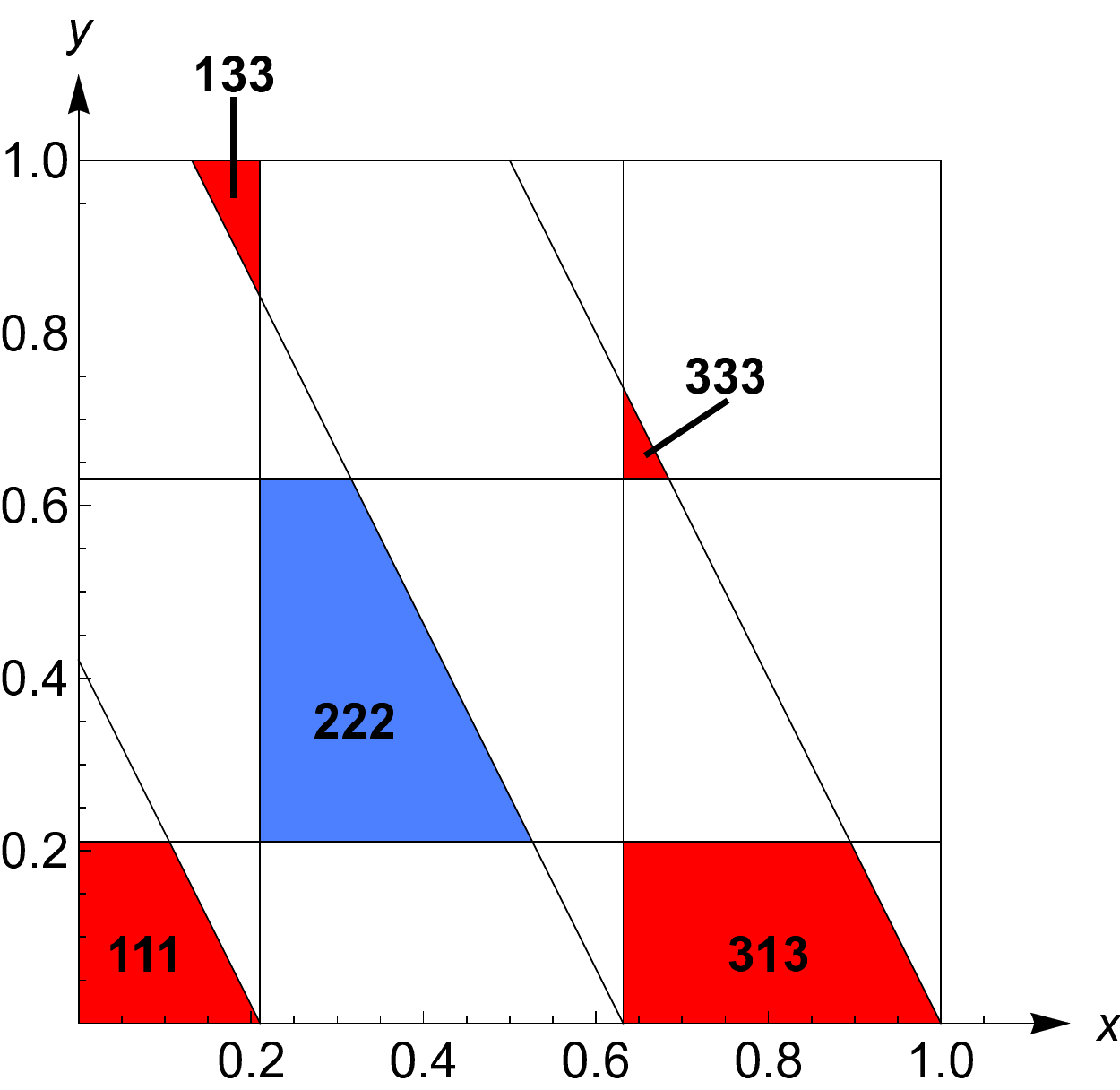}
    \qquad\qquad
    \includegraphics[width=0.4\textwidth]{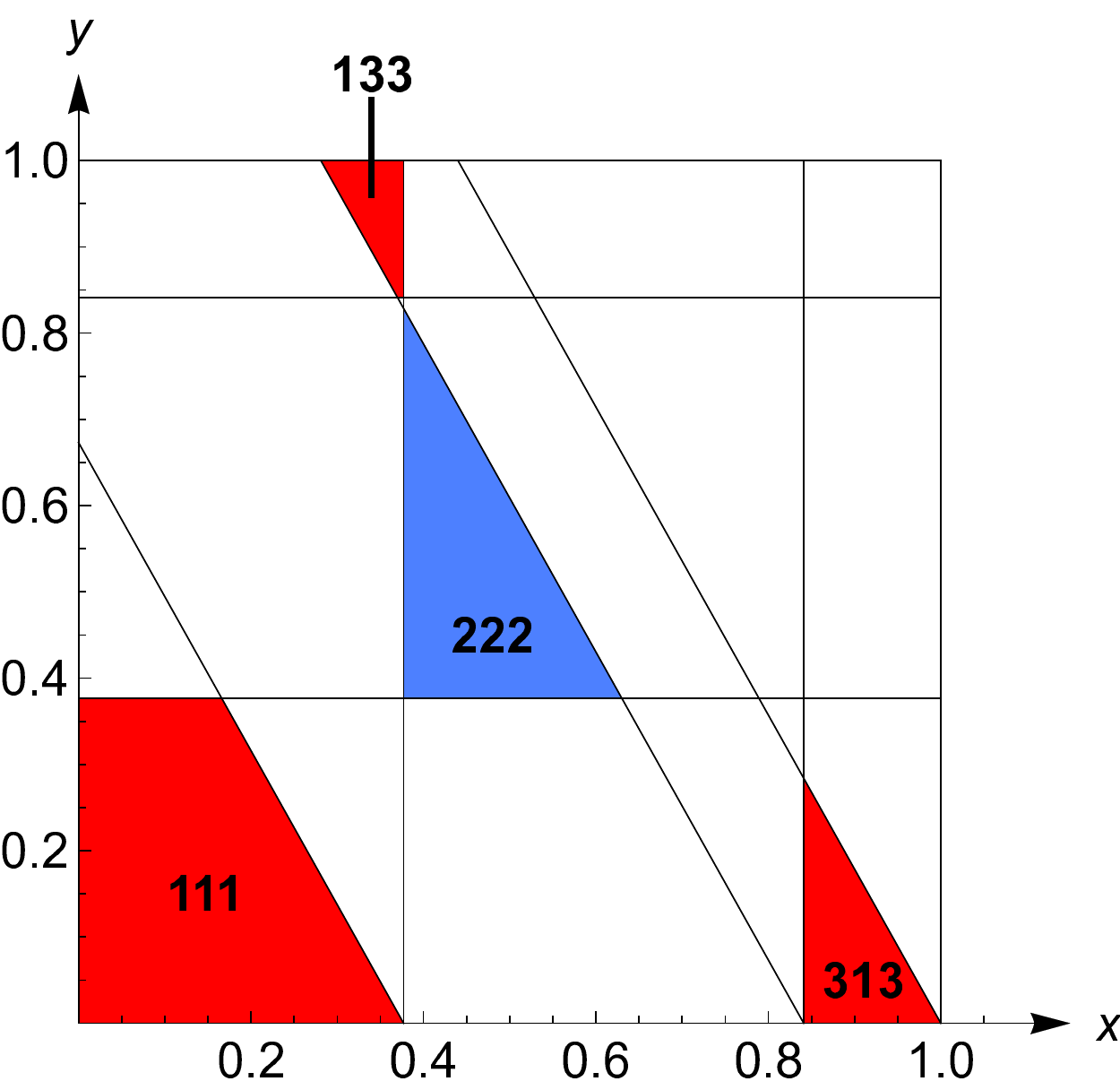}
  \end{center}
  \caption{The red and blue polygons correspond to monochromatic generalized Schur
    triples for $a=\frac12$, $s=\frac{4}{19}$, $t=\frac{12}{19}$ (left) and
    $a=0.56$, $s=0.377$, $t=0.841$ (right).}
  \label{fig:MGSTarea3}
\end{figure}

\pagebreak

\begin{theorem}\label{thm:MGSTcount0.5}
  The minimal number of monochromatic generalized Schur triples of the form
  $\bigl(x,y,x+\lfloor\frac{1}{2}y\rfloor\bigr)$ that can be attained under
  any $2$-coloring of $[n]$ of the form $R^sB^{t-s}R^{n-t}$ is given by
  \[
    \M^{(1/2)}(n) =
    \left\lfloor\frac{15n^2+72}{76}\right\rfloor +
    \begin{cases}
      1, & \text{if } n=38k+18 \text{ or } n=38k+20, \\
      -1, & \text{if } n=38k+19, \\
      0, & \text{otherwise}.
    \end{cases}
  \]
\end{theorem}
\begin{proof}
  For $a=\frac12$, it follows by \lem{min} that the optimal choice for $s$
  and~$t$ is expected around the point
  \[
    n\cdot\left(\frac{-2 a^2+a+1}{-4 a^3+5 a^2+6 a+1},\frac{-2 a^3+a^2+4 a+1}{-4 a^3+5 a^2+6 a+1}\right) =
    \left(\frac{4n}{19},\frac{12n}{19}\right).
  \]
  For this choice of parameters we end up in region $(R_{69})$ (see \fig{pwregions2}).
  Under the conditions that characterize this region, more precisely
  \[
    \frac{n}{2} \leq t \leq \frac{2n}{3} \;\land\; t - s \leq \frac{n}{2} \;\land\; 2s \leq t,
  \]
  the number of MGSTs is given by
  \begin{multline*}
    \M^{(1/2)}(n,s,t) =
    % 111
    \sum_{y=1}^s \; \sum_{x=1}^{s - \lfloor y/2\rfloor} \!\! 1 \; +
    % 222
    \sum_{y=s+1}^t \sum_{x=s+1}^{t - \lfloor y/2\rfloor} \!\! 1 \; + \;\;
    % 313
    \sum_{y=1}^s \; \sum_{x=t+1}^{n - \lfloor y/2\rfloor} \!\! 1 \; + \\
    % 133
    + \sum_{y=2t-2s+1}^n \; \sum_{x=t+1-\lfloor y/2\rfloor}^s \!\!\!\!\!\! 1 \;\; +
    % 333
    \sum_{y=t+1}^{2n-2t-1} \; \sum_{x=t+1}^{n - \lfloor y/2\rfloor} \!\! 1.
  \end{multline*}
  The five double sums correspond to the cases $111$, $222$, $313$, $133$,
  $333$, respectively, and the summation ranges are chosen such that they
  actually agree with the first two coordinates of the monochromatic triples
  in question, see \fig{MGSTarea3}.

  In order to eliminate all floor functions, a case distinction $n=38k+\ell$
  is made. It is conjectured that the minimum is attained at $(s,t)=(s_0,t_0)$
  with
  \begin{align*}
    s_0 &= \left\lfloor\frac{4n+7}{19}\right\rfloor +
    \begin{cases}
      1, & \text{if } n=19k+17, \\
      0, & \text{otherwise},
    \end{cases} \\
    t_0 &= \left\lfloor\frac{12n+6}{19}\right\rfloor +
    \begin{cases}
      1, & \text{if } n=19k+4, \\
      0, & \text{otherwise}.
    \end{cases}
  \end{align*}
  This conjecture is proven by case distinction and CAD, as in
  \thm{MGSTcount2}. As a final result, one obtains the claimed formula,
  see~\cite{KoutschanWong19} for the details.
\end{proof}

It has to be noted that all results presented so far in this section
(Theorems~\ref{thm:MGSTcount2}--\ref{thm:MGSTcount0.5}) are based on the
assumption of the optimal coloring being of the form
$R^sB^{t-s}R^{n-t}$. While we have strong evidence that this assumption is
valid for $a>1$ (and in fact we know it to be true~\cite{Schoen99} for $a=1$),
it seems to be inappropriate for $0<a<1$. More concretely,
we can construct explicit examples where we get fewer MGSTs for $a=\frac12$
than predicted in~\thm{MGSTcount0.5}: the first instance is $n=4$, where
\thm{MGSTcount0.5} yields four MGSTs for the coloring $RBBR$, namely $(1,1,1)$,
$(4,1,4)$, $(2,2,3)$, $(2,3,3)$, but where the better coloring $RBRB$ exists,
that allows only three MGSTs, namely $(1,1,1)$, $(3,1,3)$, and $(2,4,4)$.
Note, however, that this is not a counter-example to the theorem because
the coloring $RBRB$ is not of the form $R^sB^{t-s}R^{n-t}$.

We close this section by stating a conjecture about what we believe is the
true minimum for $a=\frac12$.
\begin{conjecture}
  For $n\geq12$, the minimal number of monochromatic generalized Schur triples
  of the form $\bigl(x,y,x+\lfloor\frac{1}{2}y\rfloor\bigr)$ that can be
  attained under any $2$-coloring of $[n]$ is given by
  \[
    \left\lfloor\frac{n^2 + 5}{6}\right\rfloor,
  \]
  and it occurs at the coloring $R^sB^{t-s}R^{u-t}B^{n-u}$ for
  \[
    s = \left\lfloor\frac{n + 3}{6}\right\rfloor, \qquad
    t = \left\lfloor\frac{n + 1}{2}\right\rfloor, \qquad
    u = \left\lfloor\frac{5 n + 3}{6}\right\rfloor.
  \]
\end{conjecture}

Curiously, the conjectured formula is not valid for $n=11$, where it would
give a minimum number of~$21$ MGSTs with a four-block coloring. The true
minimum is~$20$ and it is attained at the coloring $RBRBBRRBRBB$.

\vspace{0.8cm}

%%%%%%%%%%%%%%%%%%%%%%%%%%%%%%%%%%%%%%%%%%%%%%%%%%%%%%%%%%%%%%%%%%%%%%%%%%%%%%%%%%%%%%%%%%%%%%%
\section{Conclusions and outlook}
\label{sec:X}
%%%%%%%%%%%%%%%%%%%%%%%%%%%%%%%%%%%%%%%%%%%%%%%%%%%%%%%%%%%%%%%%%%%%%%%%%%%%%%%%%%%%%%%%%%%%%%%

\vspace{1mm}\noindent

In this paper we have presented, for the first time, exact formulas for the
minimum number of monochromatic (generalized) Schur triples. We give such
formulas explicitly only for the few cases $a=1,2,3,4$, but we want to point
out that we could do many more special cases, say $a=5,6,7,\dots$ or
$a=\frac32,\frac54,\dots$, based on the general analysis carried out in
\sec{real}.  In fact, the proofs would be done in completely analogous
fashion, requiring only little human interaction, but an increasing amount of
computation time.  In this sense, our paper contains a hidden treasure, which
is an infinite set of theorems that just have to be unveiled.

For future research, we propose to look more closely at the cases of
generalized Schur triples $(x,y,x+\lfloor ay\rfloor)$ with $0<a<1$. Our
analysis is based on the assumption that the optimal coloring that produces
the least number of monochromatic triples consists of three
blocks. Computational experiments suggest that this assumption is not valid
for $0<a<1$. For example, we believe that four blocks are necessary to capture
the minimum in the case $a=\frac12$, as conjectured in the previous section.
For some less nice rational numbers $a<1$ we were even not able
to detect a block pattern in the optimal coloring,
but that may be an artifact due to the limited size of~$n$ for which we can do
exhaustive searches (note that there are $2^n$ possible colorings).

\pagebreak

Our results are heavily based on symbolic computation techniques, such as
cylindrical algebraic decomposition and symbolic summation.  Often our proofs
require case distinctions into several dozens or even several hundred cases,
and it would be too tedious to check all of them by hand. The reader should be
convinced by now that symbolic computation can be very useful and that it
could be adapted to solve problems in other areas of mathematics. We
provide all details of our calculations in the supplementary electronic
material~\cite{KoutschanWong19}, which we hope is instructive for readers who
would like to become more acquainted with the techniques that we used here.

\vspace*{4mm}
\noindent
{\bf Acknowledgment.} We would like to thank our colleagues Thotsaporn
Thanatipanonda, Thibaut Verron, Herwig Hauser, and Carsten Schneider for
inspiring discussions, comments, and encouragement.  The first author was
supported by the Austrian Science Fund (FWF): P29467-N32.  The second author
was supported by the Austrian Science Fund (FWF): F5011-N15.

\vspace{0.8cm}

\bibliographystyle{plain}
\bibliography{schur}

\end{document}